\documentclass[12pt,reqno]{amsart}
\usepackage{color}
\usepackage{amsmath}

\usepackage{amsthm}
\usepackage{amssymb}
\usepackage{graphics}
\usepackage{latexsym}
\usepackage{comment}
\usepackage{ulem}

\numberwithin{equation}{section}
\newtheorem{thm}{Theorem}[section]
\newtheorem{prop}[thm]{Proposition}
\newtheorem{lem}[thm]{Lemma}
\newtheorem{cor}[thm]{Corollary}
\newtheorem{hyp}{Hypothesis}

\theoremstyle{remark}
\newtheorem{rem}{Remark}[section]
\newtheorem{defn}{Definition}

\newcommand{\BBB}{\mathbb}
\newcommand{\R}{{\BBB R}}
\newcommand{\Z}{{\BBB Z}}
\newcommand{\T}{{\BBB T}}

\newcommand{\N}{{\BBB N}}
\newcommand{\C}{{\BBB C}}
\newcommand{\HT}{{\mathcal H}}%
\newcommand{\LR}[1]{{\langle {#1} \rangle }}

\newcommand{\al}{\alpha}

\newcommand{\vp}{\varphi}

\newcommand{\e}{\varepsilon}
\newcommand{\ta}{\tau}

\newcommand{\p}{\partial}

\newcommand{\La}{\Lambda}

\newcommand{\de}{\delta}
\newcommand{\om}{\omega}
\newcommand{\Om}{\Omega}

\newcommand{\zz}{{\bf z}}

\newcommand{\supp}{\operatorname{supp}}

\newcommand{\I}{\infty}

\newcommand{\sgn}{\operatorname{sgn}}

\newcommand{\EQS}[1]{\begin{align} #1 \end{align}}
\newcommand{\EQQS}[1]{\begin{align*} #1 \end{align*}}

\newcommand{\F}{\mathcal{F}}
\newcommand{\1}{{\mathbf 1}}

\newcommand{\ti}{\widetilde}
\newcommand{\ha}{\widehat}

\title[Refined bilinear Strichartz estimates and gKdV]
{Refined bilinear Strichartz estimates with application to the well-posedness of periodic generalized KdV type equations}

\author[L. Molinet and T. Tanaka]{Luc Molinet and Tomoyuki Tanaka}
\address[L. Molinet]{Institut Denis Poisson, Universit\'e de Tours, Universit\'e d'Orl\'eans, CNRS, Parc Grandmont, 37200 Tours, France}
\email[L. Molinet]{luc.molinet@univ-tours.fr}
\address[T. Tanaka]{Graduate School of Engineering Science, Yokohama National University, Yokohama, Kanagawa, 240-8501 Japan}
\email[T. Tanaka]{tanaka-tomoyuki-fp@ynu.ac.jp}

\begin{document}
\setcounter{page}{001}

\begin{abstract}
  We improve our previous result (\cite{MT22}) on the Cauchy problem for one dimensional dispersive equations  with a quite general nonlinearity in the periodic setting.
  Under the same hypotheses that the dispersive operator behaves for high frequencies as a Fourier multiplier by $ i |\xi|^\alpha \xi $ with $ 1 \le \alpha\le 2 $, and that the nonlinear term is of the form $ \partial_x f(u) $ where $f $ is a real analytic function whose Taylor series around the origin has an infinite radius of convergence, we prove the unconditional LWP of the Cauchy problem in $H^s(\mathbb{T}) $  for $ s\ge 1-\frac{\alpha}{4} $ with $ s>1/2 $.
  It is worth noting that this result is optimal in the case $\alpha=2$ (generalized KdV equation) in view of the restriction $ s>1/2 $ for the continuous injection of $ H^s(\mathbb{T}) $ into $ L^\infty(\mathbb{T}) $.
  Our main new ingredient is the replacement of refined Strichartz estimates with refined bilinear estimates in the treatment of the worst resonant interactions.
  Such refined bilinear estimates already appeared in the work of Hani \cite{H12} in the context of  Schr\"odinger equations on a compact manifold.
  Finally, the main theorem yields global existence results for $ \alpha \in [4/3,2] $.
\end{abstract}

\maketitle

\section{Introduction}

We continue our study (\cite{MT22}) of the Cauchy problem associated with dispersive equations of the form
\EQS{
  \label{eq1}
  \partial_t u + L_{\alpha+1} u + \partial_x(f(u)) &=0,\quad (t,x)\in \R\times\T,\\
  u(0,x)&=u_0(x),\quad x\in\T, \label{initial}
}
where $ \T=\R/2\pi \Z $, under the two following hypotheses on $ L_{\al+1} $ with $ 1\le \alpha\le 2 $  and $ f$.
\begin{hyp}\label{hyp1}
$ L_{\alpha+1} $ is the Fourier multiplier operator by  $  - i p_{\al+1} $ where
$p_{\al+1}\in C^{1}(\R)\cap C^2(\R\backslash\{0\})$ is a real-valued odd function satisfying, for some $\xi_0>0$, $p_{\al+1}'(\xi)\sim \xi^\al$ and $ p_{\al+1}''(\xi)\sim \xi^{\al-1}$ for all $\xi\ge \xi_0$.
\end{hyp}
\begin{hyp}\label{hyp2}
  $f:\R\to\R$ is an analytic function whose Taylor series around the origin has an infinite radius of convergence.
\end{hyp}
Recall that this class of equations contains the famous generalized Korteweg-de Vries (gKdV) and generalized Benjamin-Ono equation (gBO) that correspond respectively to the case $ \alpha=2 $ and $ \alpha=1 $ and read respectively as
\EQQS{
  \p_t u + \p_x^3 u  + \p_x (f(u))=0
}
 and
\EQQS{
  \p_t u -  {\mathcal H} \p_x^2 u  + \p_x (f(u))=0,
}
where $ {\mathcal H} $ is the Hilbert transform (Fourier multiplier by $ -i \sgn(\xi) $).

The Cauchy problem associated with this kind of dispersive equation has been extensively studied over these last thirty years (see \cite{ABFS89, B93, GKT, HIKK, KT06, Kato83, KPV1, KPV2, KK, KT1, MST, MV15, MPV18, MPV19, Po1, Tao04}).
We refer the reader to the introduction of \cite{MT22} for a brief exposition of the main contributions.
In \cite{MT22} we proposed an approach, based on the method developed in \cite{MV15} and \cite{MPV19} to solve this Cauchy problem within the general framework described by Hypotheses \ref{hyp1}--\ref{hyp2}.
We showed in particular that the Cauchy problem associated with \eqref{eq1}--\eqref{initial} is unconditionally locally well-posed in $H^s(\T) $ with $ s\ge 1-\frac{\alpha}{2(\alpha+1)}$.

In this paper, we improve this previous result by establishing the unconditional local well-posedness of \eqref{eq1} in $ H^s(\T) $ with $ s\ge 1-\frac{\alpha}{4} $ and $ s>1/2$ for $\al\in [1,2]$.
It is worth noticing that this result is optimal in the case $ \alpha=2 $, given the restriction $s>1/2 $ for the continuous embedding of $ H^s(\T) $ into $ L^\infty(\T) $.
This also enables us to reach the energy space for $ \alpha\in [4/3,2] $ and thus extend our global existence results to this range of $ \alpha$.

Recall that Colliander, Keel, Staffilani, Takaoka, and Tao \cite{CKSTT04} showed that $k$-gKdV ($L_3=\p_x^3$ and $f(x)=x^k,k\ge 2$) is locally well-posed in $H^s(\T)$ for $s\ge 1/2$ by performing a contraction mapping argument in Bourgain's spaces (see also \cite{KS21,MR09} for well-posedness results on $k$-gBO, which corresponds to the case $L_2=-\HT\p_x^2$).
Although our result with $\al=2$ and $f(x)=x^k$ is weaker than \cite{CKSTT04} by $\e>0$ in terms of regularity, we succeed in proving the unconditional uniqueness, which ensures that the solution does not depend on how it is constructed (see Definition \ref{def_UWP} for the notion of unconditional uniqueness).

The main new ingredient in this paper is the use of refined bilinear estimates that are the bilinear counterparts of the refined Strichartz estimates introduced by \cite{KT1}.
Recall that this type of refined estimates is obtained by localizing a solution of the equation in spatial frequency, then evaluate the solution in small time intervals whose length depends on the spatial frequency, and finally summing over small time intervals to obtain an estimate on $ [0,T] $.
In the context of KdV-like equations, such estimates were first introduced by Koch-Tzvetkov \cite{KT1} in the study of the Benjamin-Ono equation.
The work \cite{KT1} was inspired by Burq, G\'erard, and Tzvetkov \cite{BGT1}, in which the authors showed Strichartz estimates for the Schr\"odinger operator on compact Riemannian manifolds without boundary.
It is well known that the Strichartz estimates on such manifolds are weaker than those in Euclidean space, but in \cite{BGT1} it is shown that one can still obtain the same Strichartz estimate as in Euclidian space, but only for small time intervals whose length depends on the spatial frequency.
This leads to Strichartz estimates with a possible loss on compact Riemannian manifolds by re-summing over the small time intervals for each frequency range and then re-summing over frequency.
(Note that this type of approach may also be used to prove Strichartz estimates for dispersive operators with variable coefficients as seen, for instance, in \cite{SS}.)
Later, Hani \cite{H12} generalized the argument of \cite{BGT1} to get bilinear Strichartz estimates on compact Riemannian manifolds without boundary for the Schr\"odinger operator.
Our refined bilinear estimates are of the same type as those obtained in \cite{H12}.
Nevertheless, as far as the authors know, this is the first time such bilinear estimates are shown to be useful in the context of equations with a derivative nonlinearity.
Indeed, one of these refined bilinear estimates enables us to improve the treatment of the worst interactions, that is, the resonant case of three high input frequencies of the same order that give rise to an output frequency of the same order.
Note that in \cite{MT22} we used refined Strichartz estimates to close our estimates in this case.
The restriction $ s\ge 1-\frac{\alpha}{4} $ follows from these resonant interactions.

 Before stating our main result, let us recall  our notion of solutions:
    \begin{defn}\label{def} Let $s>1/2$. We will say that $u\in L^\infty(]0,T[;H^s(\T)) $ is a solution to \eqref{eq1}  associated with the initial datum $ u_0 \in H^s(\T)$  if
  $ u $ satisfies \eqref{eq1}--\eqref{initial} in the distributional sense, i.e. for any test function $ \phi\in C_c^\infty(]-T,T[\times \T) $,  it
holds
  \begin{equation}\label{weakeq}
  \int_0^\infty \int_{\T} \Bigl[(\phi_t +L_{\alpha+1}\phi )u +  \phi_x f(u) \Bigr] \, dx \, dt +\int_{\T} \phi(0,\cdot) u_0 \, dx =0.
  \end{equation}
 \end{defn}
 \begin{rem} \label{rem2} Note that for $u\in L^\infty(]0,T[;H^s(\T)) $,  with $ s>1/2 $, $ f(u) $ is well-defined and
  belongs to $ L^\infty(]0,T[;H^s(\T))$. Moreover,
  Hypothesis \ref{hyp1}  forces
 \EQQS{
   L_{\alpha+1} u \in L^\infty(]0,T[;H^{s-\alpha-1}(\T)) \, .
 }
  Therefore $ u_t \in L^\infty(]0,T[; H^{s-\alpha-1}(\T)) $  and \eqref{weakeq} ensures that  \eqref{eq1} is satisfied in $ L^\infty(]0,T[; H^{s-\alpha-1}(\T)) $.
  In particular, $ u\in C([0,T]; H^{s-\alpha-1}(\T))$ and \eqref{weakeq} forces  the initial condition $ u(0)=u_0 $. Note that this ensures that
$u\in C_w([0,T];H^s(\T)) $ and thus
   $ \|u_0\|_{H_x^{s}}\le  \|u\|_{L^\infty_T H_x^s} $. Finally, we notice that this also ensures that $ u $ satisfies the Duhamel formula associated with \eqref{eq1}.
 \end{rem}
 Finally, let us recall the notion of unconditional well-posedness that
was introduced by Kato \cite{Kato95}, which is, roughly speaking, the local well-posedness with uniqueness of solutions in $L^\I (]0,T[;H^s(\T))$.

 \begin{defn}\label{def_UWP}  We will say that the Cauchy problem associated with \eqref{eq1} is unconditionally locally well-posed in $ H^s(\T )$ if for any initial data $ u_0\in H^s( {\T}) $ there exists $ T=T(\|u_0\|_{H^s})>0 $ and a solution
 $ u \in C([0,T]; H^s(\T)) $ to \eqref{eq1} emanating from $ u_0 $. Moreover, $ u $ is the unique solution to  \eqref{eq1} associated with $ u_0 $
that belongs to  $ L^\infty(]0,T[; H^s(\T) )$. Finally, for any $ R>0$, the solution-map $ u_0 \mapsto u $ is continuous from the ball of $ H^s(\T) $  with radius $ R $ centered at the origin  into $C([0,T(R)]; H^s( \T)) $.
 \end{defn}

We mention that Babin, Ilyin, and Titi \cite{BIT11} employed integration by parts in time and showed the unconditional uniqueness of the KdV equation in $L^2(\T)$.
This method is actually a normal form reduction and is successfully applied to a variety of dispersive equations (see for instance  \cite{GKO13,K19p,K22,KO12,KOY20,MPp} and references therein).

Our main result is the following one:

\begin{thm}[Unconditional well-posedness]\label{theo1}
Assume that Hypotheses \ref{hyp1}--\ref{hyp2} are satisfied
 with $ 1\le \alpha\le 2$.
 Then for any $ s\ge s(\al)=1-\frac{\al}{4}$ with $s>1/2$ the Cauchy problem associated with \eqref{eq1}--\eqref{initial} is unconditionally locally  well-posed in $ H^s(\T ) $ with a maximal time of existence $T\ge g(\|u_0\|_{H^{s(\al)}})>0 $ where $ g$ is a smooth decreasing function depending only on $ L_{\alpha+1}$ and $f$.
\end{thm}

\begin{rem} The above theorem also holds in the real line case by exactly the same approach.
 This  slightly extends  the results obtained by  the direct method making use of refined Strichartz estimate  that gives $ s>1-\frac{\al}{4}$ (see \cite{P21} and [\cite{MT22p}, Section 5]) at least in the case $ \alpha\in ]1,2[ $.
 Indeed, here we have  $ s\ge s(\al)=1-\frac{\al}{4}$ with $s>1/2$ and our hypotheses on the Fourier multiplier $ L_{\alpha+1} $ seems more general.
\end{rem}

Equation \eqref{eq1} enjoys the following conservation laws at the $ L^2$
and at the $ H^{\frac{\al}{2}} $-level:
\EQQS{
M(u)=\int_{\T} u^2 \quad \text{and} \quad
E(u)=\frac{1}{2} \int_{\T} u  \partial_x^{-1} L_{\alpha+1} u +\int_{\T}
F(u)
}
where $ \partial_x^{-1} L_{\alpha+1} $ is the Fourier multiplier by  $ \frac{p_{\alpha+1}(k)}{k} \1_{k\neq 0} $
and
\begin{equation}\label{defF} F(x) :=\int_0^x f(y) \, dy \; .
\end{equation}
At this stage, it is worth noticing that Hypothesis \ref{hyp1} ensures that the restriction of the quadratic part of the energy $ E $ to high frequencies behaves like the $ H^{\frac{\al}{2}}(\T)$-norm, whereas its restriction to the low frequencies can be controlled by the $ L^2$-norm.
Therefore, gathering these conservation laws with the above local well-posedness result, we can apply exactly the same arguments as in [\cite{MT22}, Section 5] to obtain the following GWP results for \eqref{eq1}:

\begin{cor}[Global existence for small initial data] \label{theo2}
Assume that Hypotheses \ref{hyp1}-\ref{hyp2} are satisfied with
 $ \alpha \in [4/3,2] $. Then there exists $ A=A(L_{\alpha+1},f) >0 $ such that for any initial data $ u_0\in H^s(\T) $ with $ s\ge \alpha/2 $ such that $ \|u_0\|_{H^{\frac{\al}{2}}} \le A $, the solution constructed in Theorem \ref{theo1} can be extended for all times. Moreover its trajectory is bounded in $ H^{\frac{\al}{2}}(\T) $.
\end{cor}

\begin{cor}[Global existence for arbitrary large initial data] \label{theo3}
Assume that Hypotheses \ref{hyp1}-\ref{hyp2} are satisfied with
 $ \alpha \in [4/3,2]$.
 Then the solution constructed in  Theorem \ref{theo1} can be extended for all times if the function $F$ defined
in \eqref{defF} satisfies one of the following conditions:
 \begin{enumerate}
 \item There exists $C>0  $ such that $|F(x)|\le C (1+|x|^{p+1})$ for some $ 0<p<2\alpha+1  $.
 \item  There exists $ B>0  $ such that $ F(x) \le B$ for any $x\in \R$.
 \end{enumerate}
 Moreover, its trajectory is bounded in $ H^{\frac{\al}{2}}(\T) $.
 \end{cor}

 \begin{rem}
 Typical examples for the case (1) are:
 \begin{itemize}
   \item $f(x)$ is a polynomial function of degree strictly less than $ 2\alpha +1$.
   \item $f(x)$ is a polynomial function of $ \sin(x) $ and $\cos(x) $.
 \end{itemize}
  On the other hand, typical examples for the case (2) are:
  \begin{itemize}
    \item $ f(x) $ is a  polynomial function of odd degree with $ \displaystyle\lim_{x\to +\infty} f(x)=-\infty $.
    \item $f(x) =-\exp(x) $ or $ f(x)=-\sinh(x) $.
  \end{itemize}
  \end{rem}

This paper is organized as follows.
In the next section, we introduce the notation and the function spaces and recall some basic estimates.
Section 3 is devoted to the proof of the main new ingredient of this paper that is the refined bilinear Strichartz estimate. In Sections 4 and 5, we prove the energy estimates we need on a solution and on the difference of two solutions to obtain the unconditional local well-posedness (LWP) result. In Section 6, we briefly recall how this unconditional LWP result follows from these energy estimates.
Finally, in the Appendix, for the sake of completeness, we provide the proof of the two estimates we borrowed from the framework of short-time $X^{s,b}$ spaces.

\section{Notation, Function Spaces and Basic Estimates}\label{notation}

\subsection{Notation}\label{sub_notation}

Throughout this paper, $\N$ denotes the set of non-negative integers.
For any positive numbers $a$ and $b$, we write $a\lesssim b$ when there exists a positive constant $C$ such that $a\le Cb$.
We also write $a\sim b$ when $a\lesssim b$ and $b\lesssim a$ hold.
Moreover, we denote $a\ll b$ if the estimate $b\lesssim a$ does not hold.
For two non-negative numbers $a,b$, we denote $a\vee b:=\max\{a,b\}$ and $a\wedge b:=\min\{a,b\}$.
We also write $\LR{\cdot}=(1+|\cdot|^2)^{1/2}$.
Moreover, if $a\in\R$, $a+$, respectively $a-$ denotes a number slightly greater, respectively lesser, than $a$.

For $u=u(t,x)$, $\F u=\tilde{u}$ denotes its space-time Fourier transform, whereas $\F_x u=\hat{u}$ (resp. $\F_t u$) denotes its Fourier transform in space (resp. time). We define the Riesz potentials by $D_x^s g:=\F_x^{-1}(|\xi|^s \F_x g)$.
We also denote the unitary group associated to the linear part of \eqref{eq1} by $U_\al(t)=e^{-tL_{\al+1}}$, i.e.,
\EQQS{
  U_\al(t)u=\F_x^{-1}(e^{itp_{\al+1}(\xi)}\F_x u).
}

Throughout this paper, we fix a smooth even cutoff function $\chi$:
let $\chi\in C_0^\I(\R)$ satisfy
\EQS{\label{defchi}
  0\le \chi\le 1, \quad \chi|_{-1,1}=1\quad
  \textrm{and}\quad \supp\chi\subset[-2,2].
}
We set $\phi(\xi):=\chi(\xi)-\chi(2\xi)$.
For any $l\in\N$, we define
\EQS{\label{defpsi}
  \phi_{2^l}(\xi):=\phi(2^{-l}\xi),\quad
  \psi_{2^l}(\ta,\xi):=\phi_{2^l}(\ta-p_{\al+1}(\xi)),
}
where $ip_{\al+1}(\xi)$ is the Fourier symbol of $L_{\al+1}$.
By convention, we also denote
\EQQS{
  \phi_0(\xi)=\chi(2\xi)\quad
  \textrm{and}\quad
  \psi_0(\ta,\xi)=\chi(2(\ta-p_{\al+1}(\xi))).
}
Any summations over capitalized variables such as $K, L, M$ or $N$ are presumed to be dyadic.
We work with non-homogeneous dyadic decompositions, i.e., these variables range over numbers of the form $\{2^k; k\in\N\}\cup \{0\}$.
We call those numbers \textit{non-homogeneous dyadic numbers}.
It is worth pointing out that $\sum_N\phi_N(\xi)=1$ for any $\xi\in\Z$,
\EQQS{
  \supp(\phi_N)\subset\{N/2\le |\xi|\le 2N\},\ N\ge 1,\quad
  \textrm{and}\quad
   \supp(\phi_0)\subset\{|\xi|\le 1\}.
}

Finally, we define the Littlewood--Paley multipliers $P_N$ and $Q_L$ by
\EQQS{
  P_N u=\F_x^{-1}(\phi_N \F_x u) \quad\textrm{and}
  \quad Q_Lu=\F^{-1}(\psi_L \F u).
}
We also set $P^+u:=\F_x^{-1}(\1_{\{\xi\ge 1\}}\F_x u)$,
$P^-u:=\F_x^{-1}(\1_{\{\xi\le -1\}}\F_x u)$,
$P_{\ge N}:=\sum_{K\ge N}P_K$,
$P_{\le N}:=\sum_{K\le N}P_K, Q_{\ge L}:=\sum_{K\ge L}Q_K$ and $Q_{\le L}:=\sum_{K\le L}Q_K$.

\subsection{Function Spaces}
For $1\le p\le \I$, $L^p(\T)$ is the standard Lebesgue space with the norm $\|\cdot\|_{L^p}$.

In this paper, we will use the frequency envelope method (see for instance \cite{Tao04}  and \cite{KT1}) in order to show the continuity result with respect to initial data.
To this aim, we first introduce the following:
\begin{defn}
  Let $\de>1$.
  An {\it acceptable frequency weight} $\{\om_N^{(\de)}\}_{N\in 2^\N\cup\{0\}}$ is defined as a dyadic sequence satisfying $\om_N\le \om_{2N}\le \de\om_N$ for $N\ge 1$.
  We simply write $\{\om_N\}$ when there is no confusion.
\end{defn}
With an acceptable frequency weight $\{\om_N\}$, we slightly modulate the classical Sobolev spaces in the following way:
for $s\ge0$, we define $H_\om^s(\T)$ with the norm
\EQQS{
  \|u\|_{H_\om^s}
  :=\bigg(\sum_{N\in 2^{\N}\cup\{0\}}\om_N^2 (1\vee N)^{2s}\|P_N u\|_{L^2}^2\bigg)^{\frac 12}.
}
Note that $H_\om^s(\T)=H^s(\T)$ when we choose $\om_N\equiv 1$.
Here, $H^s(\T)$ is the usual $L^2$--based Sobolev space.
If $B_x$ is one of spaces defined above, for $1\le p\le \I$ and $T>0$, we define the space-time spaces $L_t^p B_x :=L^p(\R;B_x)$ and $L_T^p B_x :=L^p([0,T];B_x)$ equipped with the norms (with obvious modifications for $p=\I$)
\EQQS{
  \|u\|_{L_t^p B_x}=\bigg(\int_\R\|u(t,\cdot)\|_{B_x}^p dt\bigg)^{\frac 1p}\quad
  \textrm{and}\quad
  \|u\|_{L_T^p B_x}=\bigg(\int_0^T\|u(t,\cdot)\|_{B_x}^p dt\bigg)^{\frac 1p},
}
respectively.
For $s,b\in\R$, we introduce the Bourgain spaces $X^{s,b}$ associated with the operator $L_{\al+1}$, endowed with the norm
\EQQS{
  \|u\|_{X^{s,b}}
  =\Bigg(\sum_{\xi=-\I}^\I \int_{-\I}^\I \LR{\xi}^{2s}\LR{\ta-p_{\al+1}(\xi)}^{2b}|\tilde{u}(\ta,\xi)|^2d\ta\Bigg)^{\frac 12}.
}
We also use a slightly stronger space $X_\om^{s,b}$ with the norm
\EQQS{
  \|u\|_{X_\om^{s,b}}
  :=\bigg(\sum_{N}\om_N^2 (1\vee N)^{2s}\|P_N u\|_{X^{0,b}}^2\bigg)^{\frac 12}.
}
In the proof of the refined bilinear Strichartz estimate, we use the Besov type $X^{s,b}$ spaces: for $b\in\R$ and $1\le q<\I$,
\EQQS{
  \|u\|_{X^{0,b,q}}:=\bigg(\sum_{L}L^{bq}\|Q_L u\|_{L_{t,x}^2}^q\bigg)^{\frac 1q}
}
with the obvious modifications in the case $q=\I$.
However, we only use $X^{0,\frac12,1}$ throughout this paper.
We define the function spaces $Z^s $ (resp. $Z^s_\om $), with $s\in \R$, as $Z^s:= L_t^\I H^s\cap X^{s-1,1}$ (resp. $Z^s_\om:= L_t^\I H_\om^s\cap X_\om^{s-1,1}$), endowed with the natural norm
\EQQS{
  \|u\|_{Z^s}=\|u\|_{L_t^\I H^s}+\|u\|_{X^{s-1,1}} \quad
  (\text{resp}.\  \|u\|_{Z^s_\om}=\|u\|_{L_t^\I H^s_\om}+\|u\|_{X^{s-1,1}_\om}) .
}
We also use the restriction in time versions of these spaces.
Let $T>0$ be a positive time and $B$ be a normed space of space-time functions.
The restriction space $B_T$ will be the space of functions $u:]0,T[\times\T\to\R$ or $\C$ satisfying
\EQQS{
  \|u\|_{B_T}
  :=\inf\{\|\tilde{u}\|_B \ |\ \tilde{u}:\R\times \T\to\R\ \textrm{or}\ \C,\ \tilde{u}=u\ \textrm{on}\ ]0,T[\times\T\}<\I.
}

Finally, we introduce a bounded linear operator from $X_{\om,T}^{s-1,1}\cap L_T^\I H_\om^s$ into $Z_\om^s$ with a bound independent of $s$ and $T$.
The existence of this operator ensures that actually $ Z^s_{\om,T}= L_T^\I H^s_\om\cap X^{s-1,1}_{\om,T}$.
Following \cite{MN08}, we define $\rho_T$ as
\EQS{\label{def_ext}
  \rho_T(u)(t):=U_\al(t)\chi(t)U_\al(-\mu_T(t))u(\mu_T(t)),
}
where $\mu_T$ is the continuous piecewise affine function defined by
\EQS{
  \mu_T(t)=
  \begin{cases}
    0 &\textrm{for}\quad t\notin]0,2T[,\\
    t &\textrm{for}\quad t\in [0,T],\\
    2T-t &\textrm{for}\quad t\in [T,2T].
  \end{cases}
}

\begin{lem}\label{extensionlem}
  Let $0<T\le 1$, $s\in\R$ and let $\{\om_N\}$ be an acceptable frequency weight.
  Then,
  \EQQS{
    \rho_T:&X_{\om,T}^{s-1,1} \cap L_T^\I H_\om^s\to Z^s_\om\\
    &u\mapsto \rho_T(u)
  }
  is a bounded linear operator, i.e.,
  \EQS{\label{eq2.1}
    \|\rho_T(u)\|_{L_t^\I H_\om^s}
    +\|\rho_T(u)\|_{X^{s-1,1}_\om}\lesssim
    \|u\|_{L_T^\I H_\om^s}
    +\|u\|_{X_{\om,T}^{s-1,1}},
  }
  for all $u\in X_{\om,T}^{s-1,1}\cap L_T^\I H_\om^s$.
  Moreover, it holds that
  \EQS{\label{eq2.1single}
  \|\rho_T(u)\|_{L_t^\I H_\om^s}
  \lesssim
  \|u\|_{L_T^\I H_\om^s}
  }
  for all $u\in L_T^\I H_\om^s$.
  Here, the implicit constants in \eqref{eq2.1} and \eqref{eq2.1single} can be chosen independent of $0<T\le 1$ and $s\in\R$.
\end{lem}

\begin{proof}
  See Lemma 2.4  in \cite{MPV19} for $ \om_N\equiv 1$ but it is obvious that the result does not depend on $ \om_N$.
\end{proof}

\subsection{Basic Estimates}

In this subsection, we collect some fundamental estimates.
Well-known estimates are adapted for our setting $H_\om^s(\T)$ and $f(u)$.

\begin{lem}
  Let $\{\om_N\}$ be an acceptable frequency weight.
  Then we have the estimate
  \EQS{\label{eq2.2}
    \|uv\|_{H_\om^s}
    \lesssim \|u\|_{H_\om^s}\|v\|_{L^\I}+\|u\|_{L^\I}\|v\|_{H_\om^s},
  }
  whenever $s>0 $ or $ s\ge 0 $ and $ \omega_N \equiv 1$.
  In particular for any fixed real smooth function $ f $ with $ f(0)=0 $, there exists a real smooth function $ G=G[f] $ that is increasing and non-negative on $ \R_+ $  such that
    \EQS{\label{eq2.2ana}
    \|f(u)\|_{H_\om^s}
    \lesssim   G(\|u\|_{L^\I}) \|u\|_{H_\om^s},
  }
   whenever $s>0 $ or $ s\ge 0 $ and $ \omega_N \equiv 1$.
 \end{lem}

\begin{proof}
  See [\cite{MT22}, Lemma 2.2].
\end{proof}

\begin{lem}
  Assume that $s_1+s_2\ge 0, s_1\wedge s_2\ge s_3, s_3<s_1+s_2-1/2$.
  Then
  \EQS{\label{eq2.3}
    \|uv\|_{H^{s_3}}\lesssim \|u\|_{H^{s_1}}\|v\|_{H^{s_2}}.
  }
  In particular, for $u,v\in H^s(\T)$ with $s>1/2 $ and any fixed real smooth function $f$, there exists a real smooth function $ G=G[f] $ that is increasing and non-negative on $ \R_+ $ such that
   \EQS{\label{eq2.3ana}
    \|f(u)-f(v)\|_{H^\theta}\le G(\|u\|_{H^{s}}+ \|v\|_{H^s})\|u-v\|_{H^{\theta}}
  }
  for $ \theta\in \{0, s-1\} $.
\end{lem}
\begin{proof}
  For \eqref{eq2.3}, see [\cite{GLM14}, Lemma 3.4].
  The proof of \eqref{eq2.3ana} can be found in [\cite{MT22}, Lemma 2.3].
\end{proof}

We will frequently use the following lemma, which can be seen as a variant of the integration by parts.

\begin{lem}\label{lem_comm1}
  Let $N\in 2^{\N}\cup\{0\}$.
  Then,
  \EQQS{
    \bigg|\int_\T \Pi(u,v)wdx\bigg|
    \lesssim\|u\|_{L_x^2}\|v\|_{L_x^2}\|\p_x w\|_{L_x^\I},
  }
  where
  \EQS{\label{def_pi}
    \Pi(u,v):=v \p_x P_N^2u +u \p_x P_N^2v .
  }
\end{lem}

\begin{proof}
  See [\cite{MT22}, Lemma 2.4].
\end{proof}

\section{Refined Bilinear Strichartz Estimates}

In this section, we establish refined bilinear Strichartz estimates which play a crucial role in our study.
In the proof of the a priori estimate (Proposition \ref{prop_apri}) and estimate for the difference (Proposition \ref{difdif}), we would like  to simultaneously use the refined bilinear Strichartz estimate \eqref{eq_bistri2} and integration by parts (Lemma \ref{lem_comm1}).
However, Lemma \ref{lem_comm1} is involved with a Fourier multiplier, which does not enable us to use \eqref{eq_bistri2} as we would like.
Therefore, we consider the symbol of $\Pi(u,v)$ (which is defined in \eqref{def_pi}) in more detail and decompose it into two parts.
See Case 3 of $A_1$ in the proof of Proposition \ref{prop_apri}.
This is the reason why we have to state the bilinear estimate \eqref{eq_bistri1.1}  with a Fourier multiplier.
For that purpose, we introduce the following notation:

\begin{defn}
  For $u,v\in L^2(\T)$ and $a\in L^\I(\R^2)$, we set
  \EQS{\label{def_lambda}
    \F_x(\La_a(u,v))(\xi)
    :=\sum_{\xi_1+\xi_2=\xi}a(\xi_1,\xi_2)
      \hat{u}(\xi_1)\hat{v}(\xi_2).
  }
  Remark that when $a\equiv1$, we have $\La_a(u,v)=uv$.
\end{defn}

Let us now state the two main results of this section:
\begin{prop}[Refined bilinear Strichartz I]\label{prop_bistri}
  Let $0<T<1$, $\al\in[1,2]$ and  $N_1,N_2\ge 1$.
 Let also  $f_1,f_2\in L^\I(]0,T[; L^2(\T))$ and let $a\in L^\I(\R^2)$ such that $\|a\|_{L^\I}\lesssim 1$.
  Finally, let $u_1,u_2\in C([0,T];L^2(\T))$ satisfying
  \EQQS{
    \p_t u_j + L_{\al+1}u_j +\p_x f_j=0
  }
  on $]0,T[\times \T$ for $j=1,2$, with $ L_{\al+1}$ satisfying Hypothesis \ref{hyp1}.
  Then
  \EQS{\label{eq_bistri1.1}
    \begin{aligned}
      \| \La_a(P_{N_1} u_1,
       P_{N_2} u_2)\|_{L_{T,x}^2}
        & \lesssim   T^{-\frac 14} \,  (N_1\vee N_2)^{\frac{1}{2}-\frac{\alpha}{4}}
         (\|P_{ N_1} u_1\|_{L^2_{T,x}}+\|P_{ N_1} f_1\|_{L^2_{T,x}})\\
         & \hspace*{20mm}\times(\|P_{N_2} u_2\|_{L_T^\I L^2_x}+\|P_{N_2} f_2\|_{L_T^\I L^2_x}).
    \end{aligned}
  }
    \end{prop}

\begin{rem}\label{rem_compari}
  In \cite{MT22}, we used the refined Strichartz estimates (Proposition 3.5 in \cite{MT22}) for resonant interactions, which roughly speaking claims
  \EQS{\label{eq_stri}
    \|P_N u\|_{L_{T,x}^4}
    \lesssim N^{\frac{1}{4(\al+1)}}
    (\|P_N u\|_{L_{T}^4 L_x^2}+\|P_N f\|_{L_{T}^4 L_x^2}),
  }
  where $u$ satisfies the assumption in Proposition \ref{prop_bistri}.
  We can view \eqref{eq_bistri1.1} as a bilinear improvement of \eqref{eq_stri} since
  \EQQS{
    \frac{2}{4(\al+1)}-\bigg(\frac 12 -\frac \al4\bigg)
    =\frac{\al(\al-1)}{4(\al+1)}\ge0
  }
  for $\al\in[1,2]$.
  Note that \eqref{eq_bistri1.1} does not improve \eqref{eq_stri} when $\al=1$.
  Recall that Theorem \ref{theo1} with $\al=1$ is exactly the same as our previous result [\cite{MT22}, Theorem 1.1] when $\al=1$.
\end{rem}

Now, it is well-known  that we can obtain a better estimate when one of two frequencies is dominantly large such as $N_1\gg N_2$.

\begin{prop}[Refined bilinear Strichartz II]\label{prop_bistri2}
  Let $0<T<1$  and $\al\in[1,2]$ and $N_1\vee N_2\gg N_1\wedge N_2\ge 1$.
  Let $f_1,f_2\in L^\I(]0,T[; L^2(\T))$ and let $a\in L^\I(\R^2)$ such that $\|a\|_{L^\I}\lesssim 1$.
   Let $u_1,u_2\in C([0,T];L^2(\T))$ satisfying
  \EQQS{
    \p_t u_j + L_{\al+1}u_j +\p_x f_j=0
  }
  on $]0,T[\times \T$ for $j=1,2$, with $ L_{\al+1}$ satisfying Hypothesis \ref{hyp1}.
  Then for any $ 0\le \theta\le 1 $ it holds
  \EQS{\label{eq_bistri2}
    \begin{aligned}
      \|\La_a(P_{N_1} u_1,&
        P_{N_2} u_2)\|_{L_{T,x}^2}  \lesssim T^{\frac{\theta-1}{2}} (N_1\wedge N_2)^{\frac \theta2} \\
     &\times(\|P_{N_1}u_1\|_{L^2_{T,x}}+\|P_{N_1}f_1\|_{L^2_{T,x}})(\|P_{N_2} u_2\|_{L_T^\infty L^2_x}+\|P_{N_2} f_2\|_{L_T^\infty L^2_x}).
    \end{aligned}
  }
\end{prop}

\begin{rem}\label{rem_f1}
  We need to introduce the parameter $ \theta>0 $ in the above proposition since a factor $ T^{-\frac{1}{2}} $ in \eqref{eq_bistri2} will not enable us
  to get a positive power of $T$ in the right hand side of the energy estimates  \eqref{P} and \eqref{eq_difdif}. See for instance Case 3 of $J_t^{A_1}$ in the proof of Proposition \ref{prop_apri}. It is worth noticing that taking $ \theta>0 $ causes a loss of a factor $N_{\textrm{min}}^{\frac \theta2}$  that is anyway allowed since we work with $ s>1/2$.
\end{rem}

\begin{rem}\label{rem_f}
  In Proposition \ref{prop_apri}, we will apply Propositions \ref{prop_bistri} and \ref{prop_bistri2} with $u_1=u_2=u$ and $f_1(t,x)=f_2(t,x)=f(u(t,x))-f(0)$, where $f$ satisfies Hypothesis \ref{hyp2}.
  Notice that $P_N f(0)=0$ when $N\ge 1$.
  This modification allows us to use \eqref{eq2.2ana} after summing over $N$.
  See \eqref{def_U} and \eqref{def_Y}.
\end{rem}

The above refined bilinear estimates are based on the following classical bilinear estimates:

\begin{prop}\label{prop_stri1_2}
 Let $\al\in[1,2]$ and $a\in L^\I(\R^2)$ with $ \|a\|_{L^\I} \le 1 $.
  Then there exists $C=C(\xi_0)>0$  such that for any real-valued functions $u_1,u_2 \in L^2(\R_t\times\T_x)$, any $N_1,N_2,L_1,L_2\ge 1$ it holds
  \EQS{\label{bil1}
    \begin{aligned}
      &\|\La_a(Q_{\le L_1} P_{N_1} u_1, Q_{\le L_2} P_{N_2} u_2)\|_{L_{t,x}^2}\\
      &\le C(L_1\wedge L_2)^{\frac 12}
        \Biggr\{\frac{(L_1\vee L_2)^{\frac 14}}{(N_1\vee N_2)^{\frac{\al-1}{4}}}+1\Biggr\}
        \|Q_{\le L_1} P_{N_1} u_1\|_{L_{t,x}^2}\|Q_{\le L_2} P_{N_2} u_2\|_{L_{t,x}^2}.
    \end{aligned}
  }
  Moreover,
  \EQS{\label{bil2}
    \begin{aligned}
      &\|\La_a(Q_{\le L_1} P_{N_1} u_1,
        Q_{\le L_2}P_{N_2}u_2)\|_{L_{t,x}^2}\\
      &\le C(L_1\wedge L_2)^{\frac 12}
      \min \Biggr(\frac{(L_1\vee L_2)^{\frac 12}}{(N_1\vee N_2)^{\frac{\al}{2}}}+1, (N_1\wedge N_2)^{\frac 12} \Biggl)\\
      &\quad\times
        \|Q_{\le L_1}P_{N_1}u_1\|_{L_{t,x}^2}
        \|Q_{\le L_2}P_{N_2}u_2\|_{L_{t,x}^2}.
    \end{aligned}
  }
 whenever $ N_1\vee N_2 \gg N_1\wedge N_2$.
 \end{prop}

To prove Proposition \ref{prop_stri1_2}, we need the two following technical lemmas.

\begin{lem}\label{lem_counting}
  Let $I$ and $J$ be two intervals on the real line and $g\in C^1(J;\R)$.
  Then
  \EQQS{
    \# \{x\in J\cap\Z;g(x)\in I\}\le\frac{|I|}{\inf_{x\in J}|g'(x)|}+1.
  }
\end{lem}

\begin{proof}
  See Lemma 2 in \cite{ST01}.
\end{proof}

\begin{lem}\label{lem_counting2}
  Let $I$ and $J$ be two intervals on $\R$ and let $\theta>0$.
  Assume that $g\in C^2(\R)$ satisfies $g''(x)\gtrsim \theta$ for $x\in J$.
  Then,
  \EQQS{
    \# \{x\in J\cap\Z; g(x)\in I\}
    \lesssim \frac{|I|^{\frac 12}}{\theta^{\frac 12}}+1
  }
\end{lem}

\begin{proof}
  Set $A:=\{x\in J\cap\Z; g(x)\in I\}$.
  We divide $\R$ into three parts:
  \EQQS{
    I_1&:=\{x\in\R; |g'(x)|\le (\theta|I|)^{\frac 12}\},\\
    I_2&:=\{x\in\R; g'(x)> (\theta|I|)^{\frac 12}\},\\
    I_3&:=\{x\in\R; g'(x)< -(\theta|I|)^{\frac 12}\}.
  }
  It is clear that $I_m\cap I_n=\emptyset$ for $m\neq n$ and $I_1\cup I_2\cup I_3=\R$.
  For $I_1$, we see from Lemma \ref{lem_counting} that
  \EQQS{
    \#(A\cap I_1)
    \lesssim \#\{x\in J\cap\Z; |g'(x)|\le (\theta|I|)^{\frac 12}\}
    \lesssim \frac{(\theta|I|)^{\frac 12}}{\inf_{x\in J} g''(x)}+1
    \lesssim \frac{|I|^{\frac 12}}{\theta^{\frac 12}}+1.
  }
  On the other hand, for $I_2$ we again use Lemma \ref{lem_counting} so that
  \EQQS{
    \#(A\cap I_2)
    &\lesssim \# \{x\in J\cap \Z; g(x)\in I, g'(x)>(\theta|I|)^{\frac 12}\}\\
    &\lesssim \frac{|I|}{(\theta|I|)^{\frac 12}}+1
     = \frac{|I|^{\frac 12}}{\theta^{\frac 12}}+1.
  }
  Similarly, we have
  \EQQS{
  \#(A\cap I_3)
  \lesssim \# \{x\in J\cap \Z; h(x)\in -I, h'(x)>(\theta|I|)^{\frac 12}\}
  \lesssim \frac{|I|^{\frac 12}}{\theta^{\frac 12}}+1,
  }
  where we put $h(x):=-g(x)$.
  This completes the proof.
\end{proof}

\begin{proof}[Proof of Proposition \ref{prop_stri1_2}]
  For simplicity, we put $v_j:= \psi_{\le L_j}\phi_{N_i} \hat{u}_j$ for $j=1,2$.
   The Plancherel theorem leads to
  \EQS{\label{no1}
    \begin{aligned}
      &\|\La_a(P_{N_1}Q_{\le L_1}u_1,P_{N_2}
        Q_{\le L_2}u_2)\|_{L_{t,x}^2}^2\\
      &=\sum_{\xi\in \Z}\int_\ta\bigg|\sum_{\xi_1\in\Z}
       \int_{\ta_1}
       a(\xi_1,\xi-\xi_1) v_1(\ta_1,\xi_1)
       v_2(\ta-\ta_1,\xi-\xi_1)d\ta_1 \bigg|^2 d\ta\\
      & \le \sum_{\xi\in\Z}\int_\ta\bigg|
       \sum_{\xi_1\in\Z}\int_{\ta_1}
       | v_1(\ta_1,\xi_1)|
       |v_2(\ta-\ta_1,\xi-\xi_1)| d\ta_1 \bigg|^2 d\ta = \| w_1 w_2 \|_{L_{t,x}^2}^2,
    \end{aligned}
 }
where $ w_i={\mathcal F}^{-1}(|v_i|) $, $ i=1,2$. Note that since  $ u_i $ is real-valued, $ |v_i | $ has to  be an even real valued function that forces $ w_i$ to be also even and real-valued.
It follows from $N_1,N_2\ge 1$ that $P_0 w_1=P_0w_2=0$.
Then we can use the trick introduced in \cite{B93} that consists in rewriting $ w_i $ as $ P^+ w_i +P^- w_i$ (see Subsection \ref{sub_notation} for the definitions of $P^+$ and $P^-$), and observing that since  $ P^- w_i(t,x) = \overline{P^+ w_i (t,x)} $, it holds
\EQS{\label{no2}
  \begin{aligned}
    &\| w_1 w_2 \|_{L_{t,x}^2}\\
    & \le    \| P^+ w_1 P^+ w_2 \|_{L_{t,x}^2}
      +\| P^- w_1 P^- w_2 \|_{L_{t,x}^2}
      +\| P^+ w_1 P^- w_2 \|_{L_{t,x}^2}\\
    &\quad +\| P^- w_1 P^+ w_2 \|_{L_{t,x}^2}
    =    4\| P^+ w_1 P^+ w_2 \|_{L_{t,x}^2}
  \end{aligned}
 }
 since $\| P^+ w_1 P^- w_2 \|_{L_{t,x}^2}=\| P^+ w_1 \overline{P^+ w_2} \|_{L_{t,x}^2}=\| P^+ w_1 P^+ w_2 \|_{L_{t,x}^2}$ and other terms involved with $P^- w_i$ was treated similarly.
 Thus,  we are reduced to working with functions with non-negative spacial frequencies $ P^+ w_1$ and  $P^+ w_2$. By using the Plancherel theorem and the Cauchy-Schwarz inequality, we get
  \EQS{\label{no3}
      \| P^+ w_1 P^+ w_2 \|_{L_{t,x}^2}^2
      &\lesssim \sup_{(\ta,\xi)\in\R\times\N}
     A_{L_1,L_2,N_1,N_2}(\ta,\xi)\|v_1\|_{L_\ta^2 l_\xi^2}^2
     \|v_2\|_{L_\ta^2 l_\xi^2}^2,
     }
    where
  \EQS{ \label{no4}
    \begin{aligned}
      &A_{L_1,L_2,N_1,N_2}(\ta,\xi)\\
      &\lesssim mes\{(\ta_1,\xi_1)\in\R\times\N; \xi-\xi_1\ge 0, \xi_1\sim N_1, \, \xi-\xi_1\sim N_2,\,\\
      &\quad\quad\LR{\ta_1-p_{\al+1}(\xi_1)}\lesssim L_1
       \quad\textrm{and}\quad
       \LR{\ta-\ta_1-p_{\al+1}(\xi-\xi_1)}\lesssim L_2\}\\
      &\lesssim (L_1\wedge L_2)\# B(\ta,\xi)
    \end{aligned}
  }
  with
  \EQQS{
    B(\ta,\xi)
    &:=\{\xi_1\ge 0;\xi-\xi_1\ge 0 ,  \xi_1\sim N_1, \,
     \xi-\xi_1\sim N_2 \; \\
    &\quad\quad\textrm{and}\
      \LR{\ta-p_{\al+1}(\xi_1)-p_{\al+1}(\xi-\xi_1)}\lesssim L_1\vee L_2\}.
  }
  We put  $g(\xi_1):=\ta-p_{\al+1}(\xi_1)-p_{\al+1}(\xi-\xi_1)$.
  When $N_1\lesssim\xi_0$ or $N_2\lesssim \xi_0$, it holds  $\# B(\ta,\xi)\lesssim \xi_0$.
   Now when $ N_1\wedge N_2 \gg \xi_0 $ then  $\xi_1, \xi-\xi_1 > \xi_0$ and by Hypothesis \ref{hyp1}  on $p_{\al+1}$, we see that $p_{\al+1}''(\xi_1),p_{\al+1}''(\xi-\xi_1)>0$
   which leads to
  \EQQS{
    |g''(\xi_1)|=|p_{\al+1}''(\xi_1) + p_{\al+1}''(\xi-\xi_1)|\gtrsim
    p_{\al+1}''(\xi_1) \vee p_{\al+1}''(\xi-\xi_1)\gtrsim  (N_1\vee N_2)^{\al-1}\; .
  }
      Lemma \ref{lem_counting2} then yields
  \EQQS{
    \# B(\ta,\xi)
    &\lesssim \#\{\xi_1\in\N; \xi_1\le \xi\ \textrm{and}\ |g(\xi_1)|\lesssim L_1\vee L_2 \}
   \lesssim \frac{(L_1\vee L_2)^{\frac 12}}{ (N_1\vee N_2)^{\frac{\al-1}{2}}} +1.
  }
This concludes the proof of \eqref{bil1} by combining the above inequality with \eqref{no1}--\eqref{no4}.

  The strategy to prove \eqref{bil2}  is similar.
  It suffices to show that
  \EQS{\label{eq_3.6}
    \#B(\ta,\xi) \lesssim \max\Bigl( \frac{L_1\vee L_2}{(N_1\vee N_2)^\al} +1, N_1\wedge N_2 \Bigr),
  }
  where
   \EQQS{
    B(\ta,\xi)
    &:=\{\xi_1\in\N;\xi-\xi_1\ge 0,\xi_1\sim N_1,\xi-\xi_1\sim N_2\ \textrm{and}\ |g(\xi_1)|\lesssim L_1\vee L_2\}.
  }
  with
   $g(\xi_1):=\ta-p_{\al+1}(\xi_1)-p_{\al+1}(\xi-\xi_1)$.

   We first notice that the estimate $ \#B(\ta,\xi)\lesssim N_1\wedge N_2 $ is direct in view of the definition of $B(\ta,\xi)$.
    Now, when $N_1\lesssim\xi_0$ or $N_2\lesssim \xi_0$, it holds $\# B(\ta,\xi)\lesssim \xi_0$ and when
  $N_1,N_2\gg\xi_0$,
  since $N_1\vee N_2  \gg N_1 \wedge N_2$, we notice that we have either $\xi_1\gg \xi-\xi_1\gg \xi_0$ or $
  \xi-\xi_1 \gg \xi_1\gg \xi_0 $.
  We also have $p_{\al+1}''(\theta)>0$ and $p_{\al+1}''(\theta)\sim \theta^{\al-1}$ for $\theta\in[\xi-\xi_1,\xi_1]$ by Hypothesis \ref{hyp1}.
  This leads to
  \EQQS{
    |g'(\xi_1)|
    =\bigg|\int_{\xi-\xi_1}^{\xi_1}
     p_{\al+1}''(\theta)d\theta\bigg|
    \sim|\xi_1^\al-(\xi-\xi_1)^\al|\gtrsim \xi_1^\al\vee (\xi-\xi_1)^\al \gtrsim (N_1\vee N_2)^\al\; .
  }
   Therefore, Lemma \ref{lem_counting} shows \eqref{eq_3.6}, which concludes the proof.
\end{proof}

Now  we have to adapt the bilinear estimates \eqref{bil1}--\eqref{bil2} to short time intervals. To do this we use the framework of the short-time $X^{s,b}$ spaces introduced by Ionescu-Kenig-Tataru \cite{IKT} (see also \cite{KTa}).
The following lemma contains the two essential estimates we need for our purpose.
We give the proof of these estimates in Appendix for the sake of completeness.

\begin{lem}\label{lem_short1}
  There exists $C>0$ such that for any $u\in X^{0,\frac 12,1}$ and $ L\ge 1 $,
  \EQS{\label{short1}
    \|\chi(L t)u\|_{L_{t,x}^2}
   \le C L^{-\frac 12}\|u\|_{X^{0,\frac 12,1}}
  }
  and
    \EQS{\label{short2}
    \|\chi(L t)u\|_{X^{0,\frac 12,1}}
    \le C\|u\|_{X^{0,\frac 12,1}},
  }
  where $ \chi$ is the smooth bump function defined in \eqref{defchi}.
\end{lem}

\begin{rem}\label{rem_short}
  Note  that \eqref{short1} is  a direct consequence of the following inequality applied with $ g= U_\al(-t) u $ : for $g\in B_{2,1}^{\frac 12}(\R)$ and $L\ge 1$
  \EQS{\label{eq_short3}
    \|\chi(L\cdot )g\|_{L^2(\R)}
    \le \|\chi(L\cdot)\|_{L^2(\R)}\|g\|_{L^\I(\R)}
    \lesssim L^{-\frac 12}\|g\|_{B_{2,1}^{\frac 12}(\R)},
  }
  where $B_{2,1}^{\frac 12}(\R)$ is the Besov space.
  \eqref{eq_short3} is reminiscent of the Heisenberg uncertainty principle.
\end{rem}

With \eqref{short1}--\eqref{short2} at hands, we can prove the following versions of the bilinear estimate \eqref{bil1}--\eqref{bil2} on short-time intervals :

\begin{prop}\label{prop_short4}
 Let $\al\in[1,2]$, $u_1,u_2\in X^{0,\frac 12,1}$, $N_1 \,, N_2 \ge 1 $, $0<T<1$ and $a\in L^\I(\R^2)$ be such that $\|a\|_{L^\I}\lesssim 1$.
 Then for $I\subset \R$ satisfying  $|I|\sim (N_1\vee N_2)^{-1} T$ it holds
  \EQS{\label{eq_3.4}
    \begin{split}
      &\|\La_a(P_{N_1} u_1, P_{N_2} u_2)\|_{L^2(I;L^2)}\\
      &\lesssim T^{\frac 14} (N_1\vee N_2)^{-\frac \al4}\|P_{N_1} u_1\|_{X^{0,\frac 12,1}}\|P_{N_2} u_2\|_{X^{0,\frac 12,1}}.
    \end{split}
  }
 Moreover, in the case $ N_1\vee N_2 \gg N_1\wedge N_2  $,  \eqref{eq_3.4}  can be refined to
  \EQS{\label{eq_3.5}
    \begin{split}
      &\|\La_a(P_{N_1}u_1,P_{N_2}u_2)\|_{L^2(I;L^2)}\\
      &\lesssim   T^{\frac \theta2} (N_1\wedge N_2)^{\frac \theta2} \, (N_1\vee N_2)^{-\frac 12}\|P_{N_1}u_1\|_{X^{0,\frac 12,1}}
        \|P_{N_2}u_2\|_{X^{0,\frac 12,1}},
    \end{split}
  }
for any $ 0\le \theta\le 1 $.
\end{prop}

\begin{proof}
By possibly replacing $ a(\cdot,\cdot) $ by $ \tilde{a} (\cdot,\cdot) $ with $ \tilde{a}(\xi_1,\xi_2)=a(\xi_2,\xi_1) $, we see that the desired estimates are symmetric in $ N_1 $ and $ N_2 $.
We may thus assume that $ N_1\ge N_2 $ and also that $I=[0,K^{-1} ]$ where we will take $K=N_1  T^{-1} $.
  Recall that $\chi|_{[-1,1]}=1$ and $\supp\chi\subset [-2,2]$ so that  we have $u_j=\chi(K t )u_j$ for $j=1,2$ on $I:=[0,K^{-1} ]$.
  We decompose $\chi(Kt)P_{N_1} u_1$ and $\chi(Kt)P_{N_2} u_2$ as
  \EQQS{
    \chi(Kt)P_{N_j} u_j
    =Q_{\le K }(\chi(Kt)P_{N_j} u_j) + \sum_{L>K }Q_{L}(\chi(Kt)P_{N_j} u_j)
  }
  for $j=1,2$.
  Then, the triangle inequality shows
  \EQQS{
    &\|\Lambda_a(P_{N_1} u_1,P_{N_2} u_2)\|_{L^2(I;L^2)}\\
    &\le \|\Lambda_a(Q_{\le K}(\chi(K t) P_{N_1}u_1),
      Q_{\le K}(\chi(K t)P_{N_2} u_2))\|_{L^2(I;L^2)}\\
    &\quad+\sum_{L_2>K}\|\Lambda_a(Q_{\le K}(\chi(K t)u_1),
      Q_{L_2}(\chi(K t) P_{N_2}u_2))\|_{L^2(I;L^2)}\\
    &\quad+\sum_{L_1>K}\|\Lambda_a(Q_{L_1}(\chi(K t)P_{N_1} u_1),
      Q_{\le K}(\chi(K t)P_{N_2} u_2))\|_{L^2(I;L^2)}\\
    &\quad+\sum_{L_1,L_2>K}\|\Lambda_a(Q_{L_1}(\chi(K t)P_{N_1}u_1),
      Q_{L_2}(\chi(K t)P_{N_2} u_2))\|_{L^2(I;L^2)}\\
    &=:A_1+A_2+A_3+A_4.
  }
  First we prove \eqref{eq_3.4}.
  For $A_1$,  Proposition \ref{prop_stri1_2} and Lemma  \ref{lem_short1} lead to
  \EQQS{
    A_1
    &\lesssim K^{\frac 12} \bigg(\frac{K^{\frac 14}}{N_1^{\frac{\al-1}{4}}}+1\bigg) \|Q_{\le K}(\chi(K t)P_{N_1}u_1)\|_{L_{t,x}^2}
      \|Q_{\le K}(\chi(K t)P_{N_2}u_2)\|_{L_{t,x}^2}
    \\
    &\lesssim \bigg(\frac{K^{-\frac 14}}{N_1^{\frac{\al-1}{4}}}+K^{-\frac 12}\bigg) \|P_{N_1} u_1\|_{X^{0,\frac 12,1}}\|P_{N_2}u_2\|_{X^{0,\frac 12,1}}\\
     & \lesssim  T^{\frac 14}\, N_1^{-\frac \al4}\|P_{N_1} u_1\|_{X^{0,\frac 12,1}}\|P_{N_2}u_2\|_{X^{0,\frac 12,1}}
  }
  since $\al\le 2$ and $ K= N_1 T^{-1} $.
  For $A_2$, we get, again with  Proposition \ref{prop_stri1_2} and Lemma  \ref{lem_short1} at hands, that
  \EQQS{
    A_2
    &\lesssim \sum_{L_2>K}
      K^{\frac 12}
      \bigg(\frac{L_2^{\frac 14}}{N_1^{\frac{\al-1}{4}}}+1\bigg)
      \|Q_{\le K}(\chi(K t)P_{N_1}u_1)\|_{L_{t,x}^2}
      \|Q_{L_2}(\chi(K t)P_{N_2}u_2)\|_{L_{t,x}^2}\\
    &\lesssim
      \bigg(\frac{K^{-\frac 14}}{N_1^{\frac{\al-1}{4}}}+ K^{-\frac 12} \bigg)  \|P_{N_1}u_1\|_{X^{0,\frac 12,1}}
      \sum_{L_2>K}  L_2^{\frac 12}
      \|Q_{L_2}(\chi(K t)P_{N_2}u_2)\|_{L_{t,x}^2}\\
&   \lesssim T^{\frac 14} \,  N_1^{-\frac \al4}\|P_{N_1}u_1\|_{X^{0,\frac 12,1}}\|P_{N_2}u_2\|_{X^{0,\frac 12,1}}.
  }
  Similarly, $A_3$ can be estimated by the same bound as above.
  Finally, we evaluate the contribution of $A_4$. Proposition \ref{prop_stri1_2} together with
 Lemma  \ref{lem_short1} lead to
  \EQQS{
    A_4
    &\lesssim \sum_{L_1, L_2>K}(L_1\wedge L_2)^{\frac 12}
      \frac{(L_1\vee L_2)^{\frac 14}}{N_1^{\frac{\al-1}{4}}}
      \prod_{j=1}^2\|Q_{L_j}(\chi(K t)P_{N_j}u_j)\|_{L_{t,x}^2}\\
    &\lesssim T^{\frac 14} \, N_1^{-\frac \al4}\|\chi( K  t)P_{N_1}u_1\|_{X^{0,\frac 12,1}}
      \|\chi(K t)P_{N_2}u_2\|_{X^{0,\frac 12,1}}\\
    &\lesssim T^{\frac 14} \,N_1^{-\frac \al4}\|P_{N_1}u_1\|_{X^{0,\frac 12,1}}\|P_{N_2}u_2\|_{X^{0,\frac 12,1}},
  }
  which completes the proof of  \eqref{eq_3.4}.

   Finally the proof of \eqref{eq_3.5} follows exactly  the same line by using the following version of \eqref{bil2}:
   \EQS{\label{bil22}
    \begin{aligned}
      &\|\La_a(Q_{\le L_1} P_{N_1} u_1,
        Q_{\le L_2}P_{N_2}u_2)\|_{L_{t,x}^2}\\
      &\le C(L_1\wedge L_2)^{\frac 12}
      \Biggr(\frac{(L_1\vee L_2)^{\frac{(1-\theta)}{2}}}{(N_1\vee N_2)^{\frac{(1-\theta)\al}{2}}}+1\Biggl)(N_1\wedge N_2)^{\frac \theta2}\\
      &\quad\times\|Q_{\le L_1}P_{N_1}u_1\|_{L_{t,x}^2}
        \|Q_{\le L_2}P_{N_2}u_2\|_{L_{t,x}^2}.
    \end{aligned}
  }
  for any $ 0\le \theta \le 1 $ whenever $ N_1\vee N_2 \gg \max(N_1 \wedge N_2 , 1) $.
   For instance, since $ N_1\ge N_2$, \eqref{bil22} leads to
  \EQQS{
    A_1
    &\lesssim  N_2^{\frac \theta2}  K^{\frac 12} \Biggr(\frac{K^{\frac{(1-\theta)}{2}}}{N_1^{\frac{(1-\theta)\al}{2}}}+1\Biggl)
     \|Q_{\le K}(\chi( K t) P_{N_1} u_1)\|_{L_{t,x}^2}
      \|Q_{\le K }(\chi(K t) P_{N_2} u_2)\|_{L_{t,x}^2}
    \\
    &\lesssim N_2^{\frac \theta2} \Biggr(\frac{(N_1^{-1}T) ^{\frac \theta2}}{N_1^{\frac{(1-\theta)\al}{2}}}+(N_1^{-1}T)^{\frac 12}\Biggl)\|P_{N_1} u_1\|_{X^{0,\frac 12,1}}\|P_{N_2} u_2\|_{X^{0,\frac 12,1}} \\
   & \lesssim T^{\frac \theta2} N_2^{\frac \theta2} (N_1^{-\frac{\al}{2}}+ N_1^{-\frac 12}) \|P_{N_1} u_1\|_{X^{0,\frac 12,1}}\|P_{N_2} u_2\|_{X^{0,\frac 12,1}}\\
    & \lesssim T^{\frac \theta2} N_2^{\frac \theta2}  N_1^{-\frac 12} \|P_{N_1} u_1\|_{X^{0,\frac 12,1}}\|P_{N_2} u_2\|_{X^{0,\frac 12,1}}
  }
   since $ 1\le \alpha\le 2 $.
   Similarly,
   \EQQS{
    &A_2\\
    &\lesssim \sum_{L_2>K}N_2^{\frac \theta2}
      K^{\frac 12}
      \Biggr(\frac{
      L_2^{\frac{(1-\theta)}{2}}}{N_1^{\frac{(1-\theta)\al}{2}}}+1\Biggl)
      \|Q_{\le K}(\chi(Kt)P_{N_1}u_1)\|_{L_{t,x}^2}
      \|Q_{L_2}(\chi(K t)P_{N_2}u_2)\|_{L_{t,x}^2}\\
    &\lesssim T^{\frac \theta2} N_2^{\frac \theta2}  (N_1^{-\frac{\al}{2}} +N_1^{-\frac 12})
     \|P_{N_1}u_1\|_{X^{0,\frac 12,1}}
     \sum_{L_2>K} L_2^{\frac 12} \|Q_{L_2}(\chi(K  t)P_{N_2}u_2)\|_{L_{t,x}^2}
    \\
&   \lesssim  T^{\frac \theta2} N_2^{\frac \theta2}  N_1^{-\frac 12} \|P_{N_1}u_1\|_{X^{0,\frac 12,1}}\|P_{N_2}u_2\|_{X^{0,\frac 12,1}}.
  }
and
 \EQQS{
    A_4
    &\lesssim \sum_{L_1, L_2>K}(L_1\wedge L_2)^{\frac 12} N_2^{\frac \theta2}
     \Biggr(\frac{(L_1\vee L_2)^{\frac{(1-\theta)}{2}}}{N_1^{\frac{(1-\theta)\al}{2}}}+1\Biggl)
      \prod_{j=1}^2\|Q_{L_j}(\chi(K t)P_{N_j}u_j)\|_{L_{t,x}^2}\\
    &\lesssim T^{\frac \theta2} N_2^{\frac \theta2}  \, N_1^{-\frac 12}\|\chi( K  t)P_{N_1}u_1\|_{X^{0,\frac 12,1}}
      \|\chi(K t)P_{N_2}u_2\|_{X^{0,\frac 12,1}}\\
    &\lesssim T^{\frac \theta2} N_2^{\frac \theta2}  \,N_1^{-\frac 12}\|P_{N_1}u_1\|_{X^{0,\frac 12,1}}\|P_{N_2}u_2\|_{X^{0,\frac 12,1}},
  }
  which completes the proof.
\end{proof}

Let us now translate \eqref{eq_3.4} and \eqref{eq_3.5} in terms of bilinear estimates for free solutions of \eqref{eq1}.

\begin{cor}\label{cor_free}
 Let $a\in L^\I(\R^2)$ be such that $\|a\|_{L^\I}\lesssim 1$, $ \alpha\in [1,2] $, $N_1,N_2\ge 1 $, $0<T<1$ and $\vp_1,\vp_2\in L^2(\T)$.
 Suppose that $I\subset \R$ satisfies $|I|\sim (N_1\vee N_2)^{-1}T$.
 Then it holds
  \EQS{\label{cor1}
    \begin{aligned}
      &\|\La_a(e^{-tL_{\al+1}}P_{N_1}\vp_1,
        e^{-tL_{\al+1}}P_{N_2}\vp_2 )\|_{L^2(I;L^2)}\\
      &\lesssim T^{\frac 14} (N_1\vee N_2)^{-\frac \al4}\|P_{N_1}\vp_1\|_{L^2}\|P_{N_2}\vp_2\|_{L^2}.
    \end{aligned}
  }
  Moreover, when $ N_1\vee N_2 \gg N_1\wedge N_2 $, \eqref{cor1} can be refined to
  \EQS{\label{cor2}
    \begin{aligned}
      &\|\La_a(e^{-tL_{\al+1}}P_{N_1}\vp_1,
        e^{-tL_{\al+1}}P_{N_2}\vp_2)
        \|_{L^2(I;L^2)}\\
      &\lesssim  T^{\frac \theta2}  (N_1\wedge N_2)^{\frac \theta2}(N_1\vee N_2)^{-\frac 12}
       \|P_{N_1}\vp_1\|_{L^2}\|P_{N_2}\vp_2\|_{L^2} \; ,
    \end{aligned}
  }
  for any $ 0 \le \theta\le 1 $.
\end{cor}

\begin{proof}
  Recall that $\ha{\chi}\in\mathcal{S}$. In particular, it holds that
  \EQS{\label{eq_3.2}
    \sum_L L^{\frac 12}\|R_{L}\chi\|_{L_t^2}\lesssim 1,
  }
  where $R_{L}\chi:=\F_t^{-1}(\phi(\ta/L)\ha{\chi}(\ta))$.
  Indeed, we have
  \EQQS{
    \|\phi(\ta/L)\ha{\chi}\|_{L_\ta^2}^2
    =\int_\R \LR{\ta}^{-8}\LR{\ta}^8
      |\phi(\ta/L)\ha{\chi}(\ta)|^2 d\ta
    \lesssim \int_\R \LR{\ta}^{-8}|\phi(\ta/L)|^2 d\ta
    \lesssim L^{-7},
  }
  which shows \eqref{eq_3.2}.
  Observe that $Q_L u=e^{-tL_{\al+1}}(R_L e^{tL_{\al+1}}u)$.
  By this and \eqref{eq_3.2}, for $j=1,2$, we obtain
  \EQS{\label{eq_3.3}
    \begin{aligned}
      \|\chi(t)e^{-tL_{\al+1}}P_{N_j}\vp_j\|_{X^{0,\frac 12,1}}
      &=\sum_L L^{\frac 12}
        \|Q_L(\chi(t)e^{-tL_{\al+1}}P_{N_j}\vp_j)\|_{L_{t,x}^2}\\
      &=\sum_L L^{\frac 12}
        \|R_L\chi\|_{L_{t}^2}\|P_{N_j}\vp_j\|_{L_{x}^2}
      \lesssim \|P_{N_j}\vp_j\|_{L_{x}^2}.
    \end{aligned}
  }
  Note that $\chi(t)=1$ for $t\in I$ since $N_1,N_2\ge 1$.
  This together with \eqref{eq_3.3} and Proposition \ref{prop_short4} with $u_j=\chi(t)e^{-tL_{\al+1}}\vp_j$ for $j=1,2$  ensure that
  \EQQS{
    &\|\La_a(e^{-tL_{\al+1}}P_{N_1}\vp_1,
      e^{-tL_{\al+1}}P_{N_2}\vp_2)\|_{L^2(I;L^2)}\\
    &\lesssim T^{\frac 14} \, (N_1\vee N_2)^{-\frac \al4}
      \|\chi(t)e^{-tL_{\al+1}}P_{N_1}\vp_1\|_{X^{0,\frac 12,1}}
      \|\chi(t)e^{-tL_{\al+1}}P_{N_2}\vp_2\|_{X^{0,\frac 12,1}}\\
    &\lesssim T^{\frac 14} \,(N_1\vee N_2)^{-\frac \al4}\|P_{N_1}\vp_1\|_{L_{x}^2}\|P_{N_2}\vp_2\|_{L_{x}^2},
  }
   and \eqref{cor2} is obtained in the same way.
\end{proof}

We are now ready to prove Propositions \ref{prop_bistri} and \ref{prop_bistri2}.
\begin{proof}[Proof of Proposition \ref{prop_bistri}]
Again, by possibly replacing $ a(\cdot,\cdot) $ by $ \tilde{a} (\cdot,\cdot) $ with $ \tilde{a}(\xi_1,\xi_2)=a(\xi_2,\xi_1) $, we can assume that $N_1\ge  N_2 $.
    We chop the time interval $[0,T]$ into small pieces of length $\sim N_1^{-1} T $, i.e., we define $\{I_{j,N_1}\}_{j\in J_{N_1}}$
     with $  \# J_{N_1}\sim N_1 $ so that $\bigcup_{j\in J_{N_1}}I_{j,N_1}=[0,T]$, $|I_{j,N_1}|\sim N_1^{-1} T$.
  For $j\in J_{N_1}$, we choose $c_{j,N_1}\in I_{j,N_1}$ at which $\|P_{N_1}u_1(t)\|_{L_x^2}^2$ attains its minimum on $I_{j,N_1}$.
  For simplicity, we write $c_j=c_{j,N_1}$.
  Since $u_1$ and $u_2$ satisfy \eqref{eq1}, on $ I_{j,N_1} $ it holds
   \EQQS{
    P_{N_1}u_1(t)
    &=e^{-(t-c_j)L_{\al+1}}P_{N_1} u_1(c_j)
     +F_{1,j},\\
    P_{N_2}u_2(t)
    &=e^{-(t-c_j)L_{\al+1}}P_{N_2} u_2(c_j)
     +F_{2,j},
  }
  where
  \EQQS{
    F_{i,j}
    :=\int_{c_j}^t e^{-(t-\tau)L_{\al+1}} P_{N_i}
     \p_x f_i (\tau) d\tau
    =\int_{c_j}^t e^{-tL_{\al+1}}  g_{N_i}(\tau ) d\tau
  }
  with
  $ g_{N_i}(\tau ):=e^{\tau L_{\al+1}}P_{N_i}\p_x f_i (\tau ) $ for $i=1,2$.
  Therefore, we have
  \EQQS{
    \|\La_a (P_{N_1} u_1,
      P_{N_2} u_2)\|_{L^2([0,T];L^2)}^2
    \le \sum_{j\in J_{N_1}} \sum_{m=1}^4 A_{m,j}^2,
  }
  where
   \EQQS{
    A_{1,j}
    &:=\|\La_a(P_{N_1} e^{-(t-c_j)L_{\al+1}} u_1(c_j), P_{N_2} e^{-(t-c_j)L_{\al+1}}u_2(c_j))\|_{L^2(I_{j,N_1};L^2)},\\
    A_{2,j}
    &:=\|\La_a(P_{N_1} e^{-(t-c_j)L_{\al+1}} u_1(c_j), F_{2,j} )\|_{L^2(I_{j,N_1};L^2)},\\
    A_{3,j}
    &:=\|\La_a(F_{1,j}, P_{N_2} e^{-(t-c_j)L_{\al+1}} u_2(c_j))\|_{L^2(I_{j,N_1};L^2)},\\
    A_{4,j}
    &:=\|\La_a(F_{1,j},F_{2,j})\|_{L^2(I_{j,N_1};L^2)}.
  }
  For the contribution of $A_{1,j}, \; j\in J_{N_1}$, \eqref{cor1} and the definition of $c_j=c_{j,N_1}$ lead to
 \EQQS{
    \sum_{j\in J_{N_1}}  A_{1,j}^2
    &\lesssim \sum_{j\in J_{N_1}}  T^{\frac 12} \, N_1^{-\frac{\al}{2}}
      \|P_{N_1} u_1(c_j)\|_{L_x^2}^2
      \|P_{N_2} u_2(c_j)\|_{L_x^2}^2\\
    &\lesssim   T^{\frac 12} \, \|u_2\|_{L_T^\I L_x^2}^2
      \sum_{j\in J_N}
       N_1^{-\frac{\al}{2}}|I_{j,N_1}|^{-1}
      \int_{I_{j,N_1}}  \|P_{N_1} u_1(t)\|_{L_x^2}^2\,
       dt\\
    &\lesssim T^{-\frac 12}N_1^{1-\frac{\al}{2}}  \|P_{N_1} u_1\|_{L^2_{T,x}}^2\|P_{N_2} u_2\|_{L_T^\I L^2_x}^2\; .
  }
   For the contribution of $A_{2,j}, \; j\in J_{N_1}$, we first notice that according to the definition \eqref{def_lambda} of $ \Lambda_a $, using  the space Fourier transform, it is easy to check that
   \EQQS{
    &\La_a\Bigl( P_{N_1} e^{-(t-c_j)L_{\al+1}} u_1(c_j), F_{2,j}\Bigr)\\
    &= \int_{c_j}^t \La_a\Bigl(P_{N_1} e^{-(t-c_j)L_{\al+1}} u_1(c_j),e^{-tL_{\al+1}}  g_{N_2}(t_2)\Bigr) \, d t_2 \;.
  }
  Then we see from the Minkowski inequality, \eqref{cor1}, \eqref{eq2.2ana} and the unitarity of $ e^{-t L_{\alpha+1}} $ in $ L^2(\T) $ that
  \EQQS{
    &\sum_{j\in J_{N_1}} A_{2,j}^2\\
    & \lesssim  \sum_{j\in J_{N_1}} \Bigl( \int_{I_j,N_1}
    \Bigl\| \La_a\Bigl(P_{N_1} e^{-(t-c_j)L_{\al+1}} u_1(c_j),e^{-tL_{\al+1}} g_{N_2}(t_2)\Bigr)\Bigr\|_{L^2(I_{j,N_1};L^2)} \, d t_2\Bigr)^2 \\
    & \lesssim    \sum_{j\in J_{N_1}}  T^{\frac 12}  N_1^{-\frac{\al}{2}}N_2^2
      \|P_{N_1} u_1(c_j) \|_{L_x^2}^2
      \Bigl(\int_{I_{j,N_1}}
      \| e^{t_2L_{\al+1}} P_{N_2} f_2(t_2) \|_{L_x^2}\,
      dt_2\bigg)^2\\
    &\lesssim T^{\frac 12}  N_1^{-\frac{\al}{2}}N_2^2 \|P_{N_2} f_2\|_{L_T^\I L^2_x}^2 (N_1^{-1} T)^2
    \sum_{j\in J_{N_1}}  |I_{j,N_1}|^{-1}
      \int_{I_{j,N_1}} \|P_{N_1} u_1(t)\|_{L_x^2}^2\, dt \\
    &\lesssim  T^{\frac{3}{2}} \, N_1^{1-\frac{\al}{2}}  \| P_{ N_1}u_1\|_{L^2_{T,x}}^2 \|P_{N_2} f_2\|_{L_T^\I L^2_x}^2.
  }
  Here, we used the definition of $c_j=c_{j,N_1}$ again in the third inequality.
  The contribution of $A_{3,j}, \; j\in J_{N_1}$, is estimated in the same way as above:
  \EQQS{
    \sum_{j\in J_{N_1}} A_{3,j}^2
    &\lesssim T^{\frac 12} \,N_1^{-\frac{\al}{2}}N_1^2 \sum_{j\in J_{N_1}}
      \|P_{N_2} u_2(c_j) \|_{L_x^2}^2
      \Bigl(\int_{I_{j,N_1}}
      \|P_{ N_1} f_1(t_1) \|_{L_x^2}\, dt_1\bigg)^2\\
    &\lesssim T^{\frac 12}  \|P_{N_2} u_2\|_{L_T^\I L^2_x}^2 N_1^{-\frac{\al}{2}}  N_1^2
    \sum_{j\in J_{N_1}}  |I_{j,N_1}|
      \int_{I_{j,N_1}} \|P_{N_1} f_1(t_1)\|_{L_x^2}^2\, dt_1 \\
    &\lesssim   T^{\frac{3}{2}} N_1^{1-\frac{\al}{2}}  \| P_{ N}f _1\|_{L^2_{T,x}}^2 \|P_{N_2} u_2\|_{L_T^\I L^2_x}^2.
  }
  Here, we used the H\"older inequality in $t_2$ in the second inequality.
  Finally, to estimate the contribution of $A_{4,j}, \; j\in J_{N_1}$, we first notice  that
  \EQQS{
    \La_a(F_{1,j},F_{2,j})= \int_{c_j}^t \int_{c_j}^t \La_a\Bigl(e^{-tL_{\al+1}}  g_{N_1}(t_1),e^{-tL_{\al+1}}  g_{N_2}(t_2)\Bigr) \, d t_2 \, d t_1
  }
   and use the Minkowski inequality together with \eqref{cor1} and  \eqref{eq2.2ana}  to get
  \EQQS{
    &\sum_{j\in J_{N_1}} A_{4,j}^2\\
    & \lesssim \sum_{j\in J_{N_1}} \Bigl( \int_{I_{j,N_1}} \int_{I_{j,N_1}}
    \Bigl\| \La_a\Bigl(e^{-tL_{\al+1}}  g_{N_1}(t_1),e^{-tL_{\al+1}}  g_{N_2}(t_2)\Bigr)\Bigr\|_{L^2(I_{j,N_1};L^2)} \, d t_2\, dt_1\Bigr)^2 \\
    & \lesssim   T^{\frac 12}   N_1^{2-\frac{\al}{2}} N_2^2   \sum_{j\in J_{N_1}} \Bigl(\int_{I_{j,N_1}}
      \|P_{N_1} f_1(t_1)\|_{L_x^2}\,  dt_1\Bigr)^2
      \Bigl(\int_{I_{j,N_1}}
      \|P_{N_2} f_2(t_2)\|_{L_x^2}\, dt_2 \bigg)^2 \\
    &\lesssim   T^{\frac 12} \|P_{N_2} f_2\|_{L_T^\I L^2_x}^2
     N_1^{2-\frac{\al}{2}}  N_2^2
    \sum_{j\in J_{N_1}}  |I_{j,N_1}|^3
      \int_{I_{j,N_1}} \|P_{N_1} f_1(t_1)\|_{L_x^2}^2\,  dt_1 \\
    &\lesssim  T^{\frac{7}{2}}  N_1^{1-\frac{\al}{2}}
     \| P_{ N_1}f_1\|_{L^2_{T,x}}^2
     \|P_{N_2} f_2\|_{L_T^\I L^2_x}^2
    }
  since $N_1\ge N_2$.
  This completes the proof.
\end{proof}

\begin{proof}[Proof of Proposition \ref{prop_bistri2}]
  The strategy of the proof is exactly the same as the one of Proposition \ref{prop_bistri} but  based on \eqref{cor2}
   instead of \eqref{cor1}.
 \end{proof}

\section{A Priori Estimate}

\subsection{Preliminary Technical Estimates}

Let us denote by $\1_T$ the characteristic function of the interval $]0,T[$.
For nonresonant interactions, we recover the derivative loss by Bourgain type estimates.
For that purpose, we first use $\1_T$ to extend a function on $[0,t]$ to a function on $\R$.
As pointed out in \cite{MV15}, $\1_T$ does not commute with $Q_L$.
Moreover, we use $X^{s-1,1}$--norm whereas $\1_T$ belongs to $H^s(\R)$ for $s<1/2$.
To avoid this difficulty, following \cite{MV15}, we decompose $\1_T$ as
\EQS{
  \1_T=\1_{T,R}^{\textrm{low}}+\1_{T,R}^{\textrm{high}},
  \quad \textrm{with} \quad
  \F_t(\1_{T,R}^{\textrm{low}})(\ta)=\chi(\ta/R)\F_t(\1_T)(\ta),
}
for some $R>0$ to be fixed later.
See also Remark 4.1 in \cite{MT22}.

In what follows, we prepare estimates and fix notation which will be used for nonresonant interactions in the proofs of Propositions \ref{prop_apri} and \ref{difdif}.

\begin{lem}[Lemma 3.5 in \cite{MPV19}]
  Let $1\le p\le \I$ and
  let $L$ be a non-homogeneous dyadic number.
  Then the operator $Q_{\le L}$ is bounded in $L_t^p L_x^2$ uniformly in $L$.
  In other words,
  \EQS{\label{eq4.3}
    \|Q_{\le L}u\|_{L_t^p L_x^2}\lesssim \|u\|_{L_t^p L_x^2},
  }
  for all $u\in L_t^p L_x^2$ and the implicit constant appearing in \eqref{eq4.3} does not depend on $L$.
\end{lem}

\begin{lem}[Lemma 3.6 in \cite{MPV19}]
  For any $R>0$ and $T>0$, it holds
  \EQS{\label{eq4.1}
    \|\1_{T,R}^{\mathrm{high}}\|_{L^1}&\lesssim T\wedge R^{-1},\\
    \label{eq4.1.1}
    \|\1_{T,R}^{\mathrm{low}}\|_{L^1}&\lesssim T
  }
  and
  \EQS{\label{eq4.2}
    \|\1_{T,R}^{\mathrm{high}}\|_{L^\I}
    +\|\1_{T,R}^{\mathrm{low}}\|_{L^\I}\lesssim 1.
  }
\end{lem}

\begin{lem}[Lemma 3.7 in \cite{MPV19}]
  Assume that $T>0$, $R>0$, and $L\gg R$.
  Then, it holds
  \EQS{\label{eq4.6}
    \|Q_L(\1_{T,R}^{\mathrm{low}}u)\|_{L_{t,x}^2}
    \lesssim \|Q_{\sim L} u\|_{L_{t,x}^2},
  }
  for all $u\in L^2(\R_t\times\T_x)$.
\end{lem}

\begin{defn}
  Let $j\in\N$.
  We define $\Om_j(\xi_1,\dots,\xi_{j+1}):\Z^{j+1}\to\R$ as
  \EQQS{
    \Om_j(\xi_1,\dots,\xi_{j+1}):=\sum_{n=1}^{j+1}p_{\al+1}(\xi_n)
  }
  for $(\xi_1,\dots,\xi_{j+1})\in\Z^{j+1}$, where $p_{\al+1}$ satisfies Hypothesis 1.
\end{defn}

\begin{lem}\label{lem_res1}
  Let $k\ge 1 $ and $(\xi_1,\dots,\xi_{k+2})\in\Z^{k+2}$ satisfy $\sum_{j=1}^{k+2}\xi_j=0$.
  Assume that $|\xi_1|\sim |\xi_2|\gtrsim |\xi_3| $ if $ k= 1 $ or $|\xi_1|\sim |\xi_2|\gtrsim |\xi_3|\gg  k  \max_{j\ge4 }|\xi_j|$ if $ k\ge 2 $.
  Then,
  \EQQS{
    |\Om_{k+1}(\xi_1,\dots,\xi_{k+2})|\gtrsim |\xi_3||\xi_1|^{\al}
  }
  for $|\xi_1|\gg (\max_{\xi\in[0,\xi_0]}|p_{\al+1}'(\xi)|)^{\frac{1}{\al}}$.
\end{lem}

\begin{proof}
  See Lemma 4.4 in \cite{MT22}.
\end{proof}

\subsection{Estimates for Solutions to \eqref{eq1}}

\begin{lem}\label{lem1}
Let $\{\om_N\}$ be an acceptable frequency weight.
Let $ 0<T<1$, $ s>1/2 $  and $ u\in L_T^\I H_\om^s $  be a solution to \eqref{eq1} associated with an initial datum
 $ u_0\in H^s_\om(\T) $.
Then $ u\in Z^s_{\om,T}$ and it holds
\begin{equation}\label{estXregular}
\|u\|_{Z^s_{\om,T}}  \lesssim \|u\|_{L^\infty_T H^s_\om} +G(\|u\|_{L^\infty_{T,x}}) \|u\|_{L^\infty_T H^{s}_\om}\; ,
\end{equation}
where $G=G[f]$ is a smooth function that is increasing and non-negative on $\R_+$.
 Moreover, for any couple $(u, v) \in   (L^\infty_T H^s)^2 $ of solutions
to \eqref{eq1} associated with a couple of initial data
 $ (u_0,v_0)\in (H^s(\T))^2 $  it holds
\begin{equation}\label{estdiffXregular}
\|u-v\|_{Z^{s-1}_{T}}  \lesssim \|u-v\|_{L^\infty_T H_x^{s-1}} +  G(\|u\|_{L^\infty_T H_x^{s}}+\|v\|_{L^\infty_T H_x^{s}})  \|u-v\|_{L^\infty_T H_x^{s-1}} \; .
\end{equation}
\end{lem}
\begin{proof}
  See Lemma 4.7 in \cite{MT22}.
\end{proof}

The following proposition is one of main estimates in the present paper.

\begin{prop}[A priori estimate]\label{prop_apri}
  Let $\{\om^{(\de)}_N\}$ be an acceptable frequency weight with $\de\le2$.
  Let $0<T<1$, $\al\in[1,2]$, and $2\ge s>1/2$ with $s\ge s(\alpha):= 1-\frac \al4$.
  Let $u\in L_T^\I H_\om^s $ be a solution to \eqref{eq1} emanating from $u_0\in H^s_\om(\T)$ on $[0,T]$.
  Then there exists a smooth function $ G=G[f] $ that is increasing and non-negative on $ \R_+ $  such that
  \EQS{\label{P}
    \|u\|_{L_T^\I H_\om^s}^2
    \le \|u_0\|_{H_\om^s}^2 + T^{\nu} G(\|u\|_{Z_{T}^{s(\al)}}+\|u\|_{Z_{T}^{\frac{1}{2}+}})
    \|u\|_{Z_{\om,T}^{s}}
    \|u\|_{L_T^\I H_\om^s}.
  }
  where $ \nu=s(\al)-1/2 $ whenever $ \alpha\in [1,2[ $ and $\nu =0+ $ for $ \alpha=2$.
\end{prop}
\begin{rem}
The strategy to show  estimate \eqref{P} (and also \eqref{eq_difdif} on the difference) is two-fold. For the non-resonant cases, we use Bourgain type estimates, which is almost identical to \cite{MT22}.
On the other hand, for the resonant cases, we apply the refined bilinear Strichartz estimates \eqref{eq_bistri1.1} and \eqref{eq_bistri2}  instead of the refined Strichartz estimates which were used in \cite{MT22}.
It is worth noticing that  \eqref{eq_bistri1.1}  does not improve linear estimates when $\al=1$ (see Remark \ref{rem_compari}).
This explains why our main result (Theorem \ref{theo1}) coincides with our previous result [\cite{MT22}, Theorem 1.1] when $\al=1$.
\end{rem}

\begin{proof} First we notice that according to Lemma \ref{lem1} it holds
$ u\in Z_{\om,T}^{s}$.
  By using \eqref{eq1}, we have
  \EQQS{
    \frac{d}{dt}\|P_N u(t,\cdot)\|_{L_x^2}^2
    =-2\int_\T P_N\p_x(f(u))P_N udx.
  }
 Fixing $ t\in ]0,T[ $,  integration in time between $0$ and $t$, multiplication by $ \om_N^2 (1\vee N)^{2s} $ and summation over $N$ yield
  \EQS{\label{PP}
  \|u(t)\|_{H_\om^s}^2
  \le \|u_0\|_{H_\om^s}^2
    +2\sum_{N\ge 1} \om_N^2 N^{2s} \bigg|\int_0^t\int_\T (f(u)-f(0)) P_N^2\p_x u dxdt'\bigg|
  }
 since $P_0\p_x u=0$.
 Now we rewrite $f(u)-f(0) $ as $ \sum_{k\ge 1} \frac{f^{(k)}(0)}{k!} u^k $
   and we notice that for any fixed $ N\in 2^{\N}  $,
    \EQS{\label{eq4.4F}
      \int_0^t\int_\T (f(u)-f(0)) P_N^2\p_x u\,  dxdt'\
   =  \sum_{k\ge 1} \frac{f^{(k)}(0)}{k!}\int_0^t\int_\T u^k P_N^2\p_x u \, dxdt' \; .
   }
   Indeed
   \EQQS{
    \begin{aligned}
       \sum_{k\ge 1} \frac{|f^{(k)}(0)|}{k!}\int_0^t\int_\T |u^k P_N^2\p_x u| \, dxdt'
       & \lesssim N  \sum_{k\ge 1} \frac{|f^{(k)}(0)|}{k!}\int_0^t \|u^k\|_{L^2_x} \|u\|_{L^2_x} dt' \\
       & \lesssim N  \sum_{k\ge 1} \frac{|f^{(k)}(0)|}{k!}\int_0^t \|u\|_{L^\I_x}^{k-1} \|u\|_{L^2_x}^2 dt' \\
     &  \lesssim N T G(\|u\|_{L_{T,x}^\I}) \|u\|_{L_T^\I L_x^2}^2<\infty ,
    \end{aligned}
  }
  that proves \eqref{eq4.4F} by Fubini-Lebesgue's theorem. \eqref{eq4.4F}
 together with Fubini-Tonelli's theorem then ensure that
    \EQS{\label{eq4.4FF}
  \sum_{N\ge 1} \om_N^2 N^{2s} \bigg|\int_0^t\int_\T (f(u)-f(0)) P_N^2\p_x
u dxdt'\bigg|
   \le \sum_{k\ge 1}  \frac{|f^{(k)}(0)|}{k!} I_{k}^t
   }
   where
  \EQQS{
    I^t_{k} := \sum_{N\ge 1} \om_N^2
     N^{2s} \bigg|\int_0^t\int_\T u^k P_N^2\p_x u \, dxdt'\bigg|.
  }
  By integration by parts it is easy to check that $ I_1^t =0 $.
  We set
  \EQS{\label{defC0}
    C_0:= \|u\|_{Z_T^{s(\al)}}+\|u\|_{Z_T^{s_0}}  \quad \text{with} \quad s_0\in ]1/2,s] \;.
  }
Let us now  prove that for any $ k\ge 1$ it holds
  \EQS{\label{eq4.4FFF}
  I_{k+1}^t \le C^k T^{\frac 14} G(C_0) C_0^k
    (\|u\|_{X_{T}^{s-1,1}}+\|u\|_{L_T^\I H_x^s})
    \|u\|_{L_T^\I H_x^s},
   }
  which  clearly  leads to \eqref{P}, taking \eqref{PP} and  \eqref{eq4.4FF} into account since  $\sum_{k\ge 1}  \frac{|f^{(k+1)}(0)|}{(k+1)!} C^k C_0^k <\infty$.

 In the sequel we  fix $ k\ge 1$.
 For simplicity, for any positive numbers $a$ and $b$, the notation $a\lesssim_k b$ means there exists a positive constant $C>0$ independent of $k$ such that
 \EQS{\label{eq_lesssim}
   a\le C^k b.
 }
  Remark that $a\le k^m b$ for $m\in\N$ can be expressed by $a\lesssim_k b$ too since an elementary calculation shows $k^m\le m! e^k$ for $m\in\N$.
  Here, $e$ is Napier's constant.
  The contribution of the sum over  $N\lesssim 1 $ in  $ I^t_{k+1} $  is easily estimated by
  \EQS{\label{eq4.4}
    \begin{aligned}
     &  \sum_{N\lesssim 1} N^{2s}
      \bigg|\int_0^t\int_\T  u^{k+1}  P_N^2\p_x u dxdt'\bigg|\\
     &\le  T\sum_{N\lesssim 1}\|u\|_{L_{T,x}^\I}^k \|u\|_{L_T^\I L_x^2}\|P_N^2 u\|_{L_T^\I L_x^2}
     \lesssim_k    T C_0^k \|u\|_{L_T^\I H^s_\om}^2.
    \end{aligned}
  }
It thus remains to bound the contribution of the sum over $N \gg 1  $ in $ I^t_{k+1} $.
  Putting
  \EQQS{
    A(\xi_1,\dots,\xi_{k+2})
    &:=\sum_{j=1}^{k+2}\phi_N^2(\xi_j)\xi_j,\\
    A_1(\xi_1,\xi_2)
    &:=\phi_N^2(\xi_1)\xi_1+\phi_N^2(\xi_2)\xi_2,\\
    A_2(\xi_3,\dots,\xi_{k+2})
    &:=\sum_{j=3}^{k+2}\phi_N^2(\xi_j)\xi_j,
  }
  so that $A(\xi_1,\dots,\xi_{k+2})=A_1(\xi_1,\xi_2)+A_2(\xi_3,\dots,\xi_{k+2})$.
  We see from the symmetry that
  \EQS{\label{eq_4.8}
  \begin{aligned}
    &\int_\T u^{k+1}P_N^2\p_x u dx\\
    &=\frac{i}{k+2}\sum_{\xi_1+\cdots+\xi_{k+2}=0}A(\xi_1,\dots,\xi_{k+2})\prod_{j=1}^{k+2}\ha{u}(\xi_j)\\
    &=\frac{i}{k+2}\sum_{N_1,\dots,N_{k+2}}\sum_{\xi_1+\cdots+\xi_{k+2}=0}A(\xi_1,\dots,\xi_{k+2})\prod_{j=1}^{k+2}\phi_{N_j}(\xi_j)\ha{u}(\xi_j).
  \end{aligned}
  }
  By symmetry  we can assume  that $ N_1\ge  N_2 \ge N_3 $ if $ k=1$,  $N_1\ge N_2 \ge N_3\ge N_4=\max_{j\ge 4}N_j  $ if $ k\ge 2 $.
  We notice that
   the cost of this choice is a constant factor  less than $(k+2)^5 $.
   It is also worth noticing that the frequency projection operator $ P_N
$ ensures that the contribution of any $ N_1\le  N/4 $ does cancel. We thus can assume that $ N_1 \ge N/4 $ and that $ N_2\gtrsim N_1/k $ with $ N_2\ge 1$.

  First, we consider the contribution of $A_2$.
  It suffices to consider the contribution of $(\phi_N(\xi_3))^2\xi_3$ since the contributions of
   $(\phi_N(\xi_j))^2\xi_j$  for $ j\ge 4 $ are clearly  simplest.
  Note that $N_3 \sim  N $ in this case.
  By the Bernstein inequality, we have
  \EQS{\label{eq4.9}
    \sum_{K}\|P_{K}u\|_{L_{T,x}^\I}
    \lesssim\sum_{K}(1\vee K)^{0-}\|u\|_{L_{T}^\I H_{x}^{s_0}}
    \lesssim\|u\|_{L_{T}^\I H_{x}^{s_0}}\lesssim C_0,
  }
  where $0-$ denotes a number slightly less than $0$ (see Subsection \ref{sub_notation}).
  We divide the contribution of $(\phi_N(\xi_3))^2\xi_3$ into two cases:
 1. $N_3\gg kN_4$ or $k=1$ and 2. $N_3\lesssim kN_4$.
  Set
  \EQQS{
  J_{t}^{A_2}
  :=\sum_{N\gg 1}\sum_{N_1,\dots,N_{k+2}} \omega_N^2 N^{2s}
    \bigg|\int_0^t\int_\T\p_x P_N^2 P_{N_3}u\prod_{j=1,j\neq3}^{k+2}P_{N_j}udxdt'\bigg|.
  }
  Note that $N\gg 1$ ensures that $N_3\gg 1$.

  \noindent
  \textbf{Case 1: $N_3\gg kN_4$ or $k=1$.}
  By impossible frequency interactions, we obtain $N_1\sim N_2$.
  In this case we make use of Lemma \ref{lem_res1} to close our estimate.
  For that purpose, we first take the extensions $\check{u}=\rho_T(u)$ of $u$ defined in \eqref{def_ext}.
  With a slight abuse of notation, we define the following functional:
  \EQS{\label{def_J2}
    J_{\I}^{A_2}(u_1,\cdots,u_{k+2})
    :=\sum_{N\gg 1}\sum_{N_1,\dots,N_{k+2}} \omega_N^2 N^{2s}
      \bigg|\int_\R\int_\T
      \p_x P_N^2 u_3\prod_{j=1,j\neq3}^{k+2}u_{j}dxdt'\bigg|.
  }
  Setting $R=N_1^{\frac{1}{3}}N_3^{\frac{4}{3}}$, we decompose $J_{t}^{A_2}$ as
  \EQQS{
    J_{t}^{A_2}
    &\le J_{\I}^{A_2}
    (P_{N_1}\1_{t,R}^{\textrm{high}}\check{u}, P_{N_2}\1_t \check{u},P_{N_3}\check{u},\cdots,P_{N_{k+2}}\check{u})\\
    &\quad+J_{\I}^{A_2}
    (P_{N_1}\1_{t,R}^{\textrm{low}}\check{u}, P_{N_2}\1_{t,R}^{\textrm{high}} \check{u},P_{N_3}\check{u},\cdots,P_{N_{k+2}}\check{u})\\
    &\quad+J_{\I}^{A_2}
    (P_{N_1}\1_{t,R}^{\textrm{low}}\check{u}, P_{N_2}\1_{t,R}^{\textrm{low}} \check{u},P_{N_3}\check{u},\cdots,P_{N_{k+2}}\check{u})\\
    &=:J_{\I,1}^{A_2}+J_{\I,2}^{A_2}+J_{\I,3}^{A_2}.
  }
  For $J_{\I,1}^{A_2}$, we see from \eqref{eq4.1} that
  $\|\1_{t,R}^{\textrm{high}}\|_{L^1}\lesssim T^{\frac 14}N_1^{-\frac 14}N_3^{-1}$, which together with \eqref{eq2.1single} gives
  \EQQS{
      J_{\I,1}^{A_2}
      &\lesssim\sum_{N_1,\dots,N_{k+2}}\om_{N_1}  \om_{N_2}
        N_3^{2s+1}\|\1_{t,R}^{\textrm{high}}\|_{L_t^1}
        \|P_{N_1}\check{u}\|_{L_t^\I L_x^2}\|P_{N_2}\check{u}\|_{L_t^\I L_x^2}
        \prod_{j=3}^{k+2}\|P_{N_j}\check{u}\|_{L_{t,x}^\I}\\
      &\lesssim_k T^{\frac 14}\|\check{u}\|_{L_t^\I H_x^{\frac{1}{2}+}}^k\|\check{u}\|_{L_t^\I H_\om^{s}}^2\sum_{N_1} N_1^{-\frac 14}
      \lesssim_k T^{\frac 14} C_0^{k} \|u\|_{L_T^\I H_\om^{s}}^2
  }
  since $N_1\ge N_2\ge N_3$.
  Here, we used \eqref{eq4.9} for $P_{N_j}\check{u}$, $j=3,\dots,k+2$.
  A similar argument with \eqref{eq4.2} yields the same bound for $J_{\I,2}^{A_2}$ as that of $J_{\I,1}^{A_2}$.
  For $J_{\I,3}^{A_2}$, we see from Lemma \ref{lem_res1} that $|\Om_{k+1}|\gtrsim N_3N_1^\al\gg R$ since $N_3\gg1$.
  Defining $L:=N_3N_1^\al$, we split $J_{\I,3}^{A_2}$ into $k+2$ parts:
  \EQQS{
    J_{\I,3}^{A_2}
    &\le J_{\I}^{A_2}(P_{N_1}Q_{\gtrsim L}(\1_{t,R}^{\textrm{low}}\check{u}),
     P_{N_2}\1_{t,R}^{\textrm{low}} \check{u},P_{N_3}\check{u},\cdots,P_{N_{k+2}}\check{u})\\
    &\quad+J_{\I}^{A_2}(P_{N_1}Q_{\ll L}(\1_{t,R}^{\textrm{low}}\check{u}),
     P_{N_2}Q_{\gtrsim L}(\1_{t,R}^{\textrm{low}} \check{u}),P_{N_3}\check{u},\cdots,P_{N_{k+2}}\check{u})\\
    &\quad+J_{\I}^{A_2}(P_{N_1}Q_{\ll L}(\1_{t,R}^{\textrm{low}}\check{u}),
     P_{N_2}Q_{\ll L}(\1_{t,R}^{\textrm{low}} \check{u}),P_{N_3}Q_{\gtrsim L}\check{u},\cdots,P_{N_{k+2}}\check{u})+\cdots\\
    &\quad
     +J_{\I}^{A_2}(P_{N_1}Q_{\ll L}(\1_{t,R}^{\textrm{low}}\check{u}),
    P_{N_2}Q_{\ll L}(\1_{t,R}^{\textrm{low}} \check{u}),P_{N_3}Q_{\ll L}\check{u},\cdots,P_{N_{k+2}}Q_{\gtrsim L}\check{u})\\
    &=:J_{\I,3,1}^{A_2}+\cdots+J_{\I,3,k+2}^{A_2}.
  }
  We also see from \eqref{eq4.1} that for $K\ge 1$
  \EQS{\label{eq4.7}
    \begin{aligned}
      \|P_{K}\1_{t,R}^{\textrm{low}}\check{u}\|_{L_{t,x}^2}
      &\le\|P_{K}\1_{t}\check{u}\|_{L_{t,x}^2}
        +\|P_{K}\1_{t,R}^{\textrm{high}}\check{u}\|_{L_{t,x}^2}\\
      &\lesssim \|P_{K}\1_{t}\check{u}\|_{L_{t,x}^2}
        +T^{\frac 14}R^{-\frac 14}\|P_{K}\check{u}\|_{L_{t}^\I L_x^2}.
    \end{aligned}
  }
  For $J_{\I,3,1}^{(2)}$, Lemma \ref{extensionlem}, the H\"older inequality, \eqref{eq4.6} and \eqref{eq4.7} imply that
  \EQQS{
      J_{\I,3,1}^{A_2}
      &\lesssim \sum_{N_1,\dots,N_{k+2}} \om_{N_1}  \om_{N_2}   N_3^{2s+1}
        \|P_{N_1}Q_{\gtrsim L} (\1_{t,R}^{\textrm{low}}\check{u})\|_{L_{t,x}^2}
        \|P_{N_2}\1_{t,R}^{\textrm{low}}\check{u}\|_{L_{t,x}^2}
        \prod_{j=3}^{k+2}\|P_{N_j}\check{u}\|_{L_{t,x}^\I}\\
      &\lesssim_k \|\check{u}\|_{L_t^\I H_x^{\frac{1}{2}+}}^k
        \sum_{N_1\gtrsim1} N_1^{2s-\al}
        \|P_{N_1} \check{u}\|_{X^{0,1}_\om}
          \|P_{\sim N_1} \1_t \check{u}\|_{L_{t}^2 H_\om^0}\\
      &\quad+T^{\frac 14}\|\check{u}\|_{L_t^\I H_x^{\frac{1}{2}+}}^{k-1}
        \sum_{N_1\gtrsim N_3} N_1^{-\al-\frac{1}{12}}N_3^{2s-\frac{1}{3}}\|P_{N_1} \check{u}\|_{X^{0,1}_\om}
          \|P_{\sim N_1} \check{u}\|_{L_{t}^2 H_\om^0}
          \|P_{N_3}\check{u}\|_{L_{t,x}^\I}\\
      &\lesssim_k T^{\frac 14}
        \|\check{u}\|_{L_t^\I H_x^{\frac{1}{2}+}}^k
        \|\check{u}\|_{L_t^\I H_\om^s}\|\check{u}\|_{X^{s-1,1}_\om}
        \lesssim_k T^{\frac 14} C_0^k
          \|u\|_{L_T^\I H_\om^s}\|u\|_{Z_{\om,T}^{s}}.
  }
  Here we used $\al\ge 1$ so that $N_1^{-\al}\le N_1^{-1}$.
  By the same way, it is easy to check that
  \EQQS{
    J_{\I,3,2}^{A_2}\lesssim_k T^{\frac 14} C_0^k
      \|u\|_{L_T^\I H_\om^s}\|u\|_{Z_{\om,T}^{s}}.
  }
  Next, we consider the contribution $J_{\I,3,3}^{A_2}$.
  Lemma \ref{extensionlem}, the H\"older inequality, the Bernstein inequality, \eqref{eq4.1.1} and \eqref{eq4.3} show
  \EQQS{
      J_{\I,3,3}^{A_2}
      &\lesssim\sum_{N_1,\dots,N_{k+2}} \om_{N_1}  \om_{N_2}  \, N_3^{2s+1}
        \|P_{N_1}Q_{\ll L}
        (\1_{t,R}^{\textrm{low}} \check{u})\|_{L_{t,x}^2}
        \|P_{N_2}Q_{\ll L}(\1_{t,R}^{\textrm{low}}\check{u})\|_{L_{t}^\I L_x^2}\\
      &\hspace*{20mm}\times\|P_{N_3}Q_{\gtrsim L}
        \check{u}\|_{L_{t}^2 L_x^\I}
        \prod_{j=4}^{k+2}\|P_{N_j}\check{u}\|_{L_{t,x}^\I}\\
      &\lesssim_k T^{\frac 12} \|\check{u}\|_{L_{t}^\I H_x^{\frac 12+}}^{k-1}
        \sum_{N_1\gtrsim N_3\ge 1} N_3^{2s}N_1^{-\al}
        \|P_{N_1}\check{u}\|_{L_{t}^\I H_\om^0}
        \|P_{\sim N_1}\check{u}\|_{L_{t}^\I H_\om^0}
        \|P_{N_3}\check{u}\|_{X^{\frac 12,1}}\\
      &\lesssim_k T^{\frac 12} C_0^{k-1}
        \sum_{N_1\gtrsim N_3\ge 1}
        N_1^{\frac{2-3\al}{4}}N_3^{2s}
        \|P_{N_1}\check{u}\|_{L_{t}^\I H_\om^0}
        \|P_{\sim N_1}\check{u}\|_{L_{t}^\I H_\om^0}
        \|P_{N_3}\check{u}\|_{X^{s(\al)-1,1}}\\
      &\lesssim_k T^{\frac 12} C_0^{k-1}
        \|\check{u}\|_{X^{s(\al)-1,1}}\|\check{u}\|_{L_t^\I H_\om^s}^2
      \lesssim_k T^{\frac 12} C_0^{k} \|u\|_{L_T^\I H_\om^s}^2
    }
    since $2-3\al<0$.
    In a similar manner, we can evaluate the contribution $J_{\I,3,j}^{A_2}$ for $j=4,\dots,k+2$ by the same bound as $J_{\I,3,3}^{A_2}$.

    \noindent
    \textbf{Case 2: $N_3\lesssim kN_4$.}
    Note that $N_4\ge 1$ since $N_3\gg 1$.
    We use Proposition \ref{prop_bistri} with $ a\equiv 1 $
    for $P_{N_2}uP_{N_4}u$ and $P_{N_1}u \p_x P_N^2 P_{N_3} u$.
    In this case we can share the lost derivative on four functions.
    For simplicity, we put
    \EQS{\label{def_U}
      U_{\ti{\om}, K}^{s,p}
      :=\|P_K u\|_{L_{T}^p H_{\ti{\om}}^s}
       +\|P_K (f(u)-f(0))\|_{L_{T}^p H_{\ti{\om}}^s}
    }
    for $s\ge 0$, $2\le p\le \I$, $\{\ti{\om}_N\}_N$ is an acceptable frequency weight.
    Recall that $H_{\ti{\om}}^s(\T)=H^s(\T)$ when $\ti{\om}_N\equiv 1$.
        By \eqref{eq4.9}, Proposition \ref{prop_bistri} with Remark \ref{rem_f},
        Young's inequality, \eqref{def_U} and \eqref{eq2.2ana}, we can see that
    \EQQS{
      J_t^{A_2}
      &\lesssim_k C_0^{k-2} \sum_{N\gg1}
       \sum_{N_1\ge N_2\gg N_3\ge N_4, \atop  N_3\lesssim kN_4}
       \omega_{N}^2  N^{2s}
       \|P_{N_1}u \p_x P_N^2 P_{N_3} u\|_{L_{T,x}^2}
       \|P_{N_2}u P_{N_4}u\|_{L_{T,x}^2}\\
      &\lesssim_k C_0^{k-2} T^{-\frac 12}
       \sum_{N_1\ge  N_2\gg N_3\ge N_4, \atop  N_3\lesssim kN_4}  \om_{N_3}^2 N_1^{\frac12-\frac\al4} N_2^{\frac12-\frac\al4}
        N_3^{2s+1}
        U_{1,N_1}^{0,2} U_{1,N_3}^{0,\I}
        U_{1,N_2}^{0,2} U_{1,N_4}^{0,\I}\\
      &\lesssim_k k^{\frac{3}{2}} C_0^{k} G(C_0) T^{-\frac 12} \|u\|_{L_T^\I H_x^{s(\al)}}
        \sum_{N_1\ge N_2}
        \Bigl(\frac{N_2}{N_1} \Bigr)^{s-s(\al)+\frac12}
        U_{\om, N_1}^{s,2} U_{\om, N_2}^{s,2}\\
      &   \lesssim_k T^{\frac 12} C_0^k G(C_0) \|u\|_{L^\infty_T H^s_\omega}^2.
      }
    Here, in the first inequality, we put $L_{T,x}^\I$ norm on $P_{N_j}u$ for $j=5,\dots,k+2$, and used \eqref{eq4.9}.
    Recall $s(\al)=1-\frac \al4$.

      Next, we consider the contribution of $A_1$.
      With a slight abuse of notation, put
      \EQQS{
        J_t^{A_1}
        :=\sum_{N\gg 1}\sum_{N_1,\dots,N_{k+2}}
          \omega_N^2 N^{2s}\bigg|\int_0^t \int_\T
          \Pi(P_{N_1}u,P_{N_2}u)
          \prod_{j=3}^{k+2} P_{N_j}u dxdt' \bigg|,
      }
      where $\Pi(u,v)$ is defined in \eqref{def_pi}.
      Note that $ P_0\Pi(P_{N_1}u,P_{N_2}u)=0 $.
      As in the estimate on the contribution of $A_2$, we divide the study of $J_t^{A_1}$ into three cases:
      1. $N_3\gg kN_4$ or $k=1$, 2. $N_2\lesssim N_3\lesssim kN_4$ and 3. $N_2\gg N_3$ and $N_3\lesssim kN_4$.
      Note that $N\gg 1$ ensures that $N_1\gg 1$.

      \noindent
      \textbf{Case 1: $N_3\gg kN_4$ or $k=1$.}
      We can use the argument of Case 1 for $A_2$ combining with  Lemma \ref{lem_comm1}.
      Since we only use Bourgain type estimates in this configuration, we do not improve the result in \cite{MT22} here. We omit the proof  since the complete proof can be found in Case 2 of the proof of Proposition 4.8 in \cite{MT22}.

      \noindent
      \textbf{Case 2: $N_2\lesssim N_3\lesssim kN_4$.}
      In this case, we have $N_3 \gtrsim N_1/k $, so that exactly the same proof as  in the Case 3 for $A_2$ is applicable to this case (and we do not need Lemma \ref{lem_comm1}).

      \noindent
      \textbf{Case 3: $N_2\gg N_3$ and $N_3\lesssim kN_4$.} In this case it holds  $ N_1\sim N_2\sim N \gg N_3 $.
      We also note that $N_3,N_4\ge 1$ since $J_t^{A_1}=0$ otherwise.
      We would like to use Lemma \ref{lem_comm1} that corresponds to integration by parts.
      The problem we meet here is how to combine Lemma \ref{lem_comm1} for $\Pi(P_{N_1}u,P_{N_2}u)$ with the refined bilinear Strichartz estimate (Proposition \ref{prop_bistri}) for $P_{N_1}uP_{N_3}u$.

      Therefore, we consider $A_1$ in more detail.
      The Taylor theorem implies that for any $ (\xi_1,\xi_2)\in \R^2 $ with $ |\xi_1|\sim |\xi_2|\sim N  $ there exists $\theta
      =\theta(\xi_1,\xi_2) \in\R$ such that $|\theta|\sim N$ and
      \EQQS{
        \phi_N^2(\xi_1)\xi_1
        =\phi_N^2(-\xi_2)(-\xi_2)
         +(\xi_1+\xi_2)(\phi_N^2(x)x)'(-\xi_2)
         +\frac{1}{2}(\xi_1+\xi_2)^2(\phi_N^2(x)x)''(\theta).
      }
               Since $\phi$ and $\phi''$ are even and $\phi'$ is odd, we have
      \EQQS{
        A_1(\xi_1,\xi_2)=\phi_N^2(\xi_1)\xi_1+\phi_N^2(\xi_2)\xi_2
        &=a_1(\xi_2)(\xi_1+\xi_2)+a_2(\xi_1,\xi_2)\frac{(\xi_1+\xi_2)^2}{N},
      }
      where
      \EQQS{
        a_1(\xi_2)
        &=\phi_N^2(\xi_2)+2\phi_N(\xi_2)\phi_N'(\xi_2)
          \frac{\xi_2}{N},\\
        a_2(\xi_1,\xi_2)
        &=2\phi_N(\theta(\xi_1,\xi_2)) \phi_N'(\theta(\xi_1,\xi_2))
         +(\phi_N'(\theta(\xi_1,\xi_2)))^2
          \frac{\theta(\xi_1,\xi_2)}{N}\\
        &\quad+ \phi_N(\theta(\xi_1,\xi_2)) \phi_N''(\theta(\xi_1,\xi_2))\frac{\theta(\xi_1,\xi_2)}{N}.
      }
      Note that the above identities forces $ a_2(\cdot,\cdot) $ to be measurable and  that $\|a_1\|_{L^\I(\R)}+\|a_2\|_{L^\I(\R^2)}\lesssim 1$.
      The advantage of this decomposition  of $A_1(\cdot,\cdot)$ is that $a_1 $ only depends on the $ \xi_2$-variable whereas we gain a factor $ \frac{\xi_1+\xi_2}{N} $ on the contribution of $ a_2(\cdot,\cdot) $ with respect to the one of $ A_1 (\cdot,\cdot) $.

By integration by parts the  term involving $ a_1$ can be rewritten as
  \EQQS{
    &\sum_{\xi_1+\cdots+\xi_{k+2}=0}
     \phi_{N_1}(\xi_1)\ha{u}(\xi_1)
     a_1(\xi_2)\phi_{N_2}(\xi_1)\ha{u}(\xi_2) (\xi_1+\xi_2) \prod_{j=3}^{k+2}\phi_{N_j}(\xi_j)\ha{u}(\xi_j)\\
    & =-\sum_{\xi_1+\cdots+\xi_{k+2}=0} \phi_{N_1}(\xi_1)\ha{u}(\xi_1)
     a_1(\xi_2)\phi_{N_2}(\xi_1)\ha{u}(\xi_2) \bigg(\sum_{i=3}^{k+2} \xi_i\bigg)  \prod_{j=3}^{k+2}\phi_{N_j}(\xi_j)\ha{u}(\xi_j) \; .
  }
To estimate  its contribution, that we will call $J_t^{a_1}$,   we make use of Proposition \ref{prop_bistri2}. For instance for the most dangerous term (with the derivative on $ u_3 $) we use    Proposition \ref{prop_bistri2} with Remark \ref{rem_f} and \eqref{eq2.2ana} on
 $\|P_{N_1} u P_{N_3} \partial_x u\|_{L^2_{T,x}} $
 and $\|P_{N_2}(\F_x^{-1}(a_1)* u) P_{N_4} u \|_{L^2_{T,x}} $
 with the trivial inequality $ \| {\mathcal F}^{-1}_x(a_1)\ast  v \|_{L^2} \lesssim \|v\|_{L^2} $, for any $ v\in L^2(\T) $, to get
 \EQQS{
  J_t^{a_1}
  &\lesssim_k C_0^{k-2} \sum_{N\gg1}\sum_{N_1\sim N_2\gg N_3\ge N_4,\atop N\sim N_1, N_3\lesssim kN_4} \om_N^2 N^{2s}
   \|P_{N_1} u P_{N_3} \partial_x u\|_{L^2_{T,x}}
   \|P_{N_2}(\check{a}_1* u) P_{N_4} u \|_{L^2_{T,x}}\\
  &\lesssim C_0^{k-2} T^{\theta-1} \sum_{N_1\sim N_2\ge N_3\ge N_4, \atop N_3\lesssim kN_4}
    \om_{N_1}^2 N_1^{2s} N_3^{\frac \theta2+1}
    N_4^{\frac \theta2} U_{1,N_1}^{0,2}U_{1,N_3}^{0,\I}
    U_{1,N_2}^{0,2}U_{1,N_4}^{0,\I}\\
  &\lesssim k^{\frac 12} C_0^{k-2} T^{\theta-1} \sum_{N_1,N_3,N_4} N_3^{-\frac \theta2}
   N_4^{-\frac \theta2} U_{\om,N_1}^{s,2}U_{1,N_3}^{\frac 12+\theta,\I}
   U_{\om,\sim N_1}^{s,2}U_{1,N_4}^{\frac 12+\theta,\I}\\
  &\lesssim T^{\theta} C_0^{k} G(C_0) \|u\|_{L_T^\I H_\om^s}^2,
 }
 where $\check{a}_1$ is the inverse Fourier transform of $a_1$ (i.e., $\F_x^{-1}(a_1)$) and $\theta\in]0,1]$.
 It is important to have a positive $\theta>0$ in order to close the estimate above.
 When $\al\in[1,2[$, we can choose $\theta=s(\al)-1/2$.
 On the other hand, when $\al=2$, we choose $\theta=\min(s_0-1/2,1)$\footnote{See \eqref{defC0} for the definition of $s_0$.}.
 Here, we used the notation $U_{\tilde{\om},N}^{s,p}$ defined in \eqref{def_U}.

 Therefore it remains to evaluate  the contribution $ J^{A_2}_t $ of $a_2(\cdot,\cdot)$, which we denote by $J^{a_2}_t$. It is to evaluate this contribution  that we need to prove   \eqref{eq_bistri1.1} with a Fourier multiplier.
 We decompose further in $ |\xi_1+\xi_2|\sim M \ge 1 $.
 Noticing that  $ N_3\lesssim k N_4  $ forces  $ M\lesssim k N_4 $,
 Proposition \ref{prop_bistri} with Remark \ref{rem_f}, \eqref{def_U}, the Young inequality and \eqref{eq2.2ana} lead to
    \EQQS{
    &J_t^{a_2}\\
    &\lesssim   \sum_{N_1\sim N_2\sim N,N_3, \dots, N_{k+2}\atop N\gg N_3\ge  N_4,1\le M\lesssim k N_4}
         \omega_{N}^2 N^{2s-1}  \bigg|\int_0^t \int_\T \p_x^2 P_M \Lambda_{a_2}(P_{N_1} u ,P_{N_2} u )
      \prod_{j=3}^{k+2}P_{N_j}u dxdt' \bigg|\\
    &\lesssim_k C_0^{k-2} \sum_{N_1\sim N_2\sim N, \atop N\gg N_3\ge N_4,1\le M\lesssim k N_4}
      \omega_{N}^2 N^{2s-1}
      \|\p_x^2 P_M \Lambda_{a_2}(P_{N_1} u ,P_{N_2} u )\|_{L_{T,x}^2}
        \|P_{N_3}  u P_{N_4} u\|_{L_{T,x}^2}\\
    &\lesssim C_0^{k-2}
     \sum_{N_1\sim N_2\sim N, \atop N\gg N_3\ge N_4, 1\le M\lesssim k N_4}
     T^{-\frac 12}
             \omega_{N}^2 N^{2s-1}M^2
        \| \Lambda_{a_2}(P_{N_1} u ,P_{N_2} u )\|_{L_{T,x}^2}
        \|P_{N_3}  u P_{N_4} u\|_{L_{T,x}^2}\\
    &\lesssim  C_0^{k-2} T^{-\frac 12}  \sum_{N\gg N_3\ge N_4}
     \sum_{1\le M\lesssim k N_4}
       M^2 N^{2s-\frac 12-\frac \al4}  N_3^{\frac 12- \frac \al4}
     U_{1,N}^{0,2} U_{1,\sim N}^{0,\I}
     U_{1,N_3}^{0,2} U_{1,N_4}^{0,\I}\\
    &\lesssim  T^{-\frac 12} k^2 C_0^{k} G(C_0) \|u\|_{L_T^\I H_\om^s}
     \sum_{N\gg N_3}
       \Bigl( \frac{N_3}{N}\Bigr)^{\frac 12+\frac \al4}
       U_{\om,N}^{s,2}U_{1,N_3}^{s(\al),2}\\
    &\lesssim_k T^{\frac 12} C_0^{k} G(C_0)
       \|u\|_{L_T^\I H^s_\om}^2.
    }
    This completes the proof.
  \end{proof}

  \section{Estimate for The Difference}

  We provide the estimate (at the regularity $s-1$) for the difference $w$ of two solutions $u,v$ of \eqref{eq1}. In this section, we do not use the frequency envelope, so we always argue on the standard Sobolev space $H^s(\T)$.

  \begin{prop}\label{difdif}
    Let $0<T<1$, $\al\in[1,2]$ and $2\ge s>1/2 $ with $ s \ge s(\al):=1-\frac \al4 $.
    Let $u$ and $v$ be two solutions of \eqref{eq1} belonging to $Z^s_T$ and associated with the inital data $u_0\in H^s(\T)$ and $v_0\in H^s(\T)$, respectively.
    Then there exists a smooth function $ G=G[f] $ that is
  increasing  and non-negative on $ \R_+ $  such that
    \EQS{\label{eq_difdif}
      \|w\|_{L_T^\I H_x^{s-1}}^2
      \le \|u_0-v_0\|_{H_x^{s-1}}^2 +T^{\nu}G(\|u\|_{Z^s_T}+\|v\|_{Z^s_T})
      \|w\|_{Z^{s-1}_T}
      \|w\|_{L_T^\I H_x^{s-1}},
    }
    where we set $w=u-v$ and $ \nu=\min(s-1/2,1/4)$.

  \end{prop}

  \begin{proof}
    According to Lemma \ref{lem1}, we notice that $u,v\in Z_T^s$.
    Observe that $w$ satisfies
    \EQS{\label{eq_w}
      \p_t w+L_{\al+1}w=-\p_x(f(u)-f(v))
    }
   Rewriting $ f(u)-f(v) $ as
   \EQQS{
     f(u)-f(v)=\sum_{k\ge 1} \frac{f^{(k)}(0)}{k!} (u^k-v^k)=\sum_{k\ge 1} \frac{f^{(k)}(0)}{k!}w \sum_{i=0}^{k-1} u^i  v^{k-1-i}
   }
   and arguing as in the proof of Proposition \ref{prop_apri}, we see from \eqref{eq_w} that for $t\in[0,T]$
    \EQQS{
      \|w(t)\|_{H_x^{s-1}}^2
      \le \|w_0\|_{H_x^{s-1}}^2
        +  2\sum_{k\ge 1}    \frac{|f^{(k)}(0)|}{(k-1)!}  \max_{i\in \{0,..,k-1\}} I_{k,i}^t ,
    }
    where $w_0=u_0-v_0$ and
    \EQQS{
      I_{k,i}^t:= \sum_{N\ge 1}  N^{2(s-1)} \bigg|\int_0^t\int_\T  u^i  v^{k-1-i} wP_N^2\p_x w\,  dxdt'\bigg| \; .
    }
    It is clear that $I_{1,i}^t=0$ by the integration by parts.
    Therefore we are reduced to estimating the contribution of
    \EQS{\label{Ik}
    I_{k+1}^t=  \sum_{N\ge 1}  N^{2(s-1)} \bigg|\int_0^t\int_\T  \textbf{z}^k wP_N^2\p_x w\,  dxdt'\bigg| \;
    }
    where $\textbf{z}^k $ stands for $ u^i v^{k-i} $ for some $i\in \{0,..,k\}
  $.  We set
    \EQQS{
      C_0:=\|u\|_{Z^s_T}+\|v\|_{Z^s_T}.
    }
    We claim  that for any $ k\ge 1$ it holds
    \EQS{\label{eq4.4FFFdif}
    I_{k+1}^t \le C^k T^{\frac 14} G(C_0)C_0^k
      \|w\|_{Z^{s-1}_T}
      \|w\|_{L_T^\I H^{s-1}_x},
     }
    which clearly leads to \eqref{P}, taking \eqref{PP} and  \eqref{eq4.4FF} into account since  $\sum_{k\ge 1}  \frac{|f^{(k+1)}(0)|}{k!} C^k C_0^k <\infty$.

   In the sequel we fix $ k\ge 1 $ and we estimate $ I_{k}^t $.
   We also use the notation $a\lesssim_k b$ defined in \eqref{eq_lesssim}.
   The contribution of the sum over  $N\lesssim 1$ in \eqref{Ik}  is easily
  estimated thanks to \eqref{eq2.2} by
    \EQQS{
        &\sum_{N\lesssim1} (1\vee N)^{2(s-1)}\bigg|\int_0^t\int_\T \zz^{k}wP_N^2\p_x w dxdt'\bigg|\\
        &\lesssim T\sum_{N\lesssim 1}\|w\|_{L_T^\I H_x^{s-1}}
        \|\zz^k P_N^2\p_x w\|_{L_T^\I H_x^{1-s}}
        \lesssim_k T \|w\|_{L_T^\I H_x^{s-1}}^2,
    }
    since $s>1/2$. In the last inequality, we used $1-s<1/2$.
    Therefore, in what follows, we can assume that $N\gg 1.$
    A similar argument to \eqref{eq_4.8} yields
    \EQQS{
      &\sum_{N\gg1} N^{2(s-1)}\bigg|\int_0^t\int_\T \zz^{k}wP_N^2\p_x w dxdt'\bigg|\\
      &\le \sum_{N\gg1}\sum_{N_1,\dots,N_{k+2}}N^{2(s-1)}\bigg|\int_0^t\int_\T \Pi(P_{N_1}w,P_{N_2}w)\prod_{j=3}^{k+2}P_{N_j}z_j dxdt'\bigg|
      =:J_t,
    }
    where $\Pi(f,g)$ is defined by \eqref{def_pi} and $ z_i\in \{u,v\} $ for $ i\in\{3,..,k+2\} $.
    By  symmetry, we may assume that $N_1\ge N_2$. Moreover, we may assume that $ N_3\ge N_4 $ for $k=2 $ and  $N_3\ge N_4\ge N_5=\max_{j\ge 5}N_j$ for $ k\ge 3$. Note again that
     the cost of this choice is a constant factor  less than $(k+2)^{5} $. It
  is also worth noticing that the frequency projectors in  $\Pi(\cdot,\cdot)$ ensure that $ N_1\sim N $ or $ N_2\sim N $ and in particular
     $ N_1\gtrsim N $. We also remark that we can assume that $ N_3\ge 1 $ since the contribution of $ N_3=0 $ does vanish by integration by parts. Finally we note that we can also assume that $ N_2\ge 1 $ since in the case $ N_2=0 $ we must have $ N_3 \gtrsim N_1/k $ and it is easy to check that by the Young inequality
  \EQQS{
  J_t & \lesssim k T\|w\|_{L^\infty_T H_x^{s-1}} \|z_3 \|_{L^\infty_T H_x^s}  \|P_0 w\|_{L^\I_{T,x}}
  \prod_{j=4}^{k+2} \|z_j\|_{L^\infty_T H_x^{s}}
  \lesssim_k T \|w\|_{L^\infty_T H_x^{s-1}}^2 \; .
  }

    We consider the following contribution to  $J_t$:
    \begin{itemize}
      \item $N_4\gtrsim N_1/k  $ ($k\ge 2$) ,
      \item $N_1\gg k N_4$ and $N_2\gtrsim N_3$ (or $k=1 $ and $N_2\gtrsim N_3$) ,
      \item $N_1\gg k N_4$ and $N_2\ll N_3$ (or $k=1 $ and $N_2\ll N_3$).
    \end{itemize}

    \noindent
    \textbf{Case 1: $N_4\gtrsim N_1/k  $.}
    Note that $N_3,N_4\ge 1$ since $N_1\gtrsim N\gg1$.
    In a similar manner to \eqref{def_U}, we define
    \EQS{
      \label{def_W}
      W_K^{s,p}
      &:= \|P_{K}w\|_{L_T^p H_x^s}
       +\|P_{K}(f(u)-f(v))\|_{L_T^p H_x^s},\\
      \label{def_Y}
      Y_{N_j}^{s,p}
      &:=\|P_{N_j}z_j\|_{L_T^p H_x^s}
       +\|P_{N_j}(f(z_j)-f(0))\|_{L_T^p H_x^s}
    }
    for $s\in\R$, $2\le p\le \I$, $K\ge 1$ and $j=3,\dots, k+2$.
    Proposition \ref{prop_bistri} together with Remark \ref{rem_f}, \eqref{def_W}, \eqref{def_Y}, \eqref{eq2.3ana} and the Young inequality lead to
    \EQQS{
      J_t
      &\lesssim_k C_0^{k-2}
       \sum_{N, N_2\le N_1 \lesssim k N_4\le k  N_3}
        N^{2s-2} \Bigl(\| \partial_x P_N^2 P_{N_1} w P_{N_2} w \|_{L^2_{T,x}}\\
      &\hspace*{45mm}+\| P_{N_1} w \partial_x P_N^2 P_{N_2} w \|_{L^2_{T,x}}\Bigr)
      \| P_{N_3} z_3 P_{N_4} z_4\|_{L^2_{T,x}} \\
      &\lesssim C_0^{k-2} T^{-\frac 12}
  \sum_{N_2\le N_1 \lesssim k N_4\le k  N_3}
        N_1^{2s-1} N_1^{\frac 12-\frac{\alpha}{4}}
        N_3^{\frac 12-\frac \al4}
        W_{N_1}^{0,2} W_{N_2}^{0,\I}
        Y_{N_3}^{0,2} Y_{N_4}^{0,\I} \\
      & \lesssim_k C_0^{k} T^{-\frac 12}
  \, G(C_0)
       \|w\|_{L_T^\I H_x^{s-1}}
        \sum_{N_1\lesssim k  N_3}
        \Bigr(\frac{N_1}{k N_3}\Bigr)^{\frac 12 -}
        W_{N_1}^{s-1,2} Y_{N_3}^{s(\al),2} \\
      &\lesssim_k T^{\frac 12} C_0^{k}  G(C_0)\|w\|_{L^\infty_T H_x^{s-1}}^2
    }
    since $s>1/2$.

    \noindent
    \textbf{Case 2: $N_1\gg k N_4$ and $N_2\gtrsim N_3$ (or $ k=1 $ and $N_2\gtrsim  N_3$).}
    The contribution of $J_t$ in this case can be estimated by the same way as the contribution of $A_1$ in Proposition \ref{prop_apri},
    replacing $N^{2s}$, $P_{N_1}u$, $P_{N_2}u$, $P_{N_j}u$ for $j=3,\dots,k+2$ by $N^{2(s-1)}$, $P_{N_1}w$, $P_{N_2}w$, $P_{N_j}z_j$ for $j=3,\dots,k+2$, respectively.

    \noindent
    \textbf{Case 3: $N_1\gg k N_4$ and $N_2\ll N_3$ (or $ k=1 $ and $N_2\ll N_3$).}
   Note that in this case $N_1\sim N_3\sim N \gg N_2\vee N_4$.
   We further divide the contribution of $J_t$ into two cases:
    \begin{itemize}
      \item $k N_4\gtrsim  N_2$,
      \item $ k N_4 \ll  N_2 $ or $k=1$.
    \end{itemize}
   However, it suffices to consider the first case $k N_4\gtrsim  N_2$ since the second case is exactly the same as Subcase 3.1 in the proof of Proposition 5.1 in \cite{MT22}.
   Its result requires only $s>1/2$ since we can use Bourgain type estimates in this configuration, which is sufficient for our purpose.
   Now, we treat the case $k N_4\gtrsim N_2$.
   Notice that $N_3,N_4\ge 1$.
   We can apply Proposition \ref{prop_bistri2} two times with $ \theta=\min(s-1/2,1/2) $ (see also Remark \ref{rem_f})  to get
    \EQQS{
      J_t
      &\lesssim_k C_0^{k-2}
       \sum_{N\sim N_1\sim N_3 \ge N_4\gtrsim  N_2/k \atop N_1\gg N_4, N_3\gg N_2}
      N^{2s-2}  \Bigl( \| \partial_x P_N^2 P_{N_1} w P_{N_4} z_4 \|_{L_{T,x}^2}  \| P_{N_3} z_3 P_{N_2} w\|_{L_{T,x}^2}\\
      & \hspace*{45mm}
      + \| P_{N_1} w P_{N_4} z_4 \|_{L_{T,x}^2}
        \| P_{N_3} z_3 \partial_x P_N^2 P_{N_2} w\|_{L_{T,x}^2}\Bigr) \\
      &\lesssim T^{\theta-1} C_0^{k-2}
        \sum_{N_1 \gtrsim N_4\gtrsim  N_2/k}    N_2^{\frac \theta2} N_4^{\frac \theta2} N_1^{2s-1}
        W_{N_1}^{0,2} W_{N_2}^{0,\I}
        Y_{\sim N_1}^{0,2} Y_{N_4}^{0,\I} \\
      &\lesssim_k   T^{\theta-1}\, k^\frac{1+\theta}{2} C_0^{k-1} G(C_0) \|w\|_{L_T^\I H_x^{s-1}}
       \sum_{ N_1 \gtrsim N_4} N_4^{-(s \wedge (2s-1))} N_4^\theta  \,
       W_{N_1}^{s-1,2} Y_{\sim N_1}^{s,2}
       Y_{N_4}^{s,\I}. \\
       & \lesssim_k  T^{\theta} C_0^{k} G(C_0)\|w\|_{L^\infty_T  H_x^{s-1}}^2
    }
  since $ s>1/2$ and $ \theta=\min(s-1/2,1/2) $.
   Here, we used the Cauchy-Schwarz inequality to sum on $ N_1$.
   Note that the above estimate only requires $s>1/2$.
  \end{proof}

  \section{Local and Global Well-Posedness Results}
  \subsection{Local Well-Posedness}
  With Lemma \ref{lem1}, Propositions \ref{prop_apri} and \ref{difdif} at hands, the proofs of Theorem \ref{theo1} and Corollaries \ref{theo2}--\ref{theo3} follow exactly the same lines as in \cite{MT22}.
  For instance, to prove the unconditional uniqueness, we take  $ u_0\in H^s(\T) $ with $ s\ge s(\al) $ and $ s>1/2 $  and $u,v$ two solutions to the Cauchy problem \eqref{eq1} emanating from $u_0$ that belong to $L^\infty_T H^s $.
    According to Lemma \ref{lem1}, we know that
    $u,v\in Z^s_T $ and  Proposition \ref{difdif} together with \eqref{estdiffXregular}  ensure that $ u\equiv v $ on $[0,T_0] $ with $ 0<T_0\le T $ that only depends on $\|u\|_{Z^{s}_T}+\|v\|_{Z^{s}_T} $.
    Therefore $ u(T_0)=v(T_0) $ and we can reiterate the same argument on $ [T_0,T] $. This proves that $u\equiv v $ on $[0,T] $  after a finite number of iteration.

    Now the existence of a solution in $ H^s(\T) $ follows also from Lemma \ref{lem1}, Propositions \ref{prop_apri} and \ref{difdif} by constructing a sequence of smooth solutions associated with a smooth approximation of $ u_0\in H^s(\T) $.
    Since this sequence of solutions is bounded in $ L^\I(]0,T[;H^s(\T)) $ for some $ T>0 $, depending only on $ \|u_0\|_{H^{s_0}} $ with $ s_0 =\max\big(s(\alpha),\frac{1}{2}+\big) $, it is a Cauchy sequence in $ L^\I(]0,T[;H^{s'}) $ for any $s'<s $.
    We can pass to the limit and prove that this sequence converges in some sense to a solution  $u\in L^\I(]0,T[;H^s(\T)) $ to \eqref{eq1} emanating from $u_0 $.
    The continuity of $ u $ with values in $ H^s(\T) $ follows from classical argument involving the reversibility and time translation invariance of the equation together with the estimate  \eqref{P} whereas the continuity of the flow map follows from the frequency envelope argument introduced in \cite{KT1}.
    See also Remark 4.2 in \cite{MT22}.

   \subsection{Global Existence Results}
   The proofs of  Corollaries \ref{theo2} and \ref{theo3}  are exactly the same as the ones of Theorem 1.2 and Theorem 1.3 in \cite{MT22}.
   The improvements with respect to these last results are only due to the improvement of the local well-posedness result.
   We thus omit the proof here and refer to \cite{MT22}.

\section*{Appendix}
In this appendix, we provide the proof of Lemma \ref{lem_short1} (see also \cite{IKT}).

\begin{proof}[Proof of Lemma \ref{lem_short1}]
  We see from \eqref{eq_short3} with $g=U_\al(-t)u$ that
  \EQQS{
    \|\chi(Lt)u\|_{L_{t,x}^2}
    =\|\chi(Lt)U_{\al}(-t)u\|_{L_{t,x}^2}
    \lesssim L^{-\frac12}\|U_{\al}(-t)u\|_{L_x^2 (B_{2,1}^{\frac12})_t}
    \lesssim L^{-\frac12}\|u\|_{X^{0,\frac12,1}},
  }
  which shows \eqref{short1}.
  In particular, we have $\|Q_{\le L}(\chi(Lt)u)\|_{L_{t,x}^2}
  \le \|\chi(Lt)u\|_{L_{t,x}^2}
  \lesssim L^{-\frac12}\|u\|_{X^{0,\frac12,1}}$.
  Now we show \eqref{short2}.
  From the definition \eqref{defpsi} of $\psi $ it is easy to check that $ \|\psi_{\le L}(\ta,\xi)\|_{L_\ta^2}\lesssim L^{\frac 12}$ so that
    \EQS{\label{eq_3.1}
     \|\psi_{L_1}(\ta,\xi)
        \LR{\ta-p_{\al+1}(\xi)}^{-\frac 12}\|_{L_{\ta}^2}
      \lesssim L_1^{-\frac 12}\|\psi_{L_1}\|_{L_\ta^2}
      \lesssim 1 .
    }
    We also notice that $\F_t(\chi(Lt))(\ta)=L^{-1}\ha{\chi}(L^{-1}\ta)$.
  By definition, we have
  \EQQS{
    &\|\chi(L(\cdot))u\|_{X^{0,\frac 12,1}}\\
    &\lesssim
     L^{\frac 12} \|Q_{\le L}(\chi(Lt)u)\|_{L_{t,x}^2}
     + \sum_{L_1>L} L_1^{\frac 12} \|Q_{L_1}(\chi(Lt)u)\|_{L_{t,x}^2}\\
    &\lesssim \|u\|_{X^{0,\frac 12,1}} + \sum_{L_1>L} L_1^{\frac 12}
      \bigg\|\psi_{L_1}(\ta,\xi)
      \int_\R L^{-1}\ha{\chi}(L^{-1}(\ta-\ta'))
      \ti{u}(\ta',\xi)d\ta'\bigg\|_{L_\ta^2 l_\xi^2}\; .
  }
It thus remains to show that
  \EQS{\label{tyty}
    \sum_{L_1> L}L_1^{\frac 12}\bigg\|\psi_{L_1}(\ta,\xi)
      \int_\R L^{-1}|\ha{\chi}(L^{-1}(\ta-\ta'))\ti{u}(\ta',\xi)|d\ta'\bigg\|_{L_\ta^2 l_\xi^2}
    \lesssim \|u\|_{X^{0,\frac 12,1}},
  }
  which will complete the proof of \eqref{short2}.
  The mean value theorem implies that
  \EQQS{
    |\psi_{L_1}(\ta,\xi)|
    &=|\check{\psi}_{L_1}(\ta,\xi)\psi_{L_1}(\ta,\xi)|\\
    &\le |\check{\psi}_{L_1}(\ta,\xi)|
      |\psi_{L_1}(\ta,\xi)-\psi_{L_1}(\ta',\xi)|
      +|\check{\psi}_{L_1}(\ta,\xi)\psi_{L_1}(\ta',\xi)|\\
    &\lesssim L_1^{-1}|\check{\psi}_{L_1}(\ta,\xi)||\ta-\ta'|
      +|\psi_{L_1}(\ta',\xi)|
  }
  since $\chi\in C_0^\I(\R)$ and $0\le \chi \le 1$ (see the definition of $\psi $ in \eqref{defpsi}).
  By using this, we are reduced to bounding $ A$ and $ B$ defined by
  \EQQS{
    A&:=\sum_{L_1> L}L_1^{\frac 12}
      \|(\psi_{L_1}(\cdot,\xi)|\ti{u}(\cdot,\xi)|)
        *_\ta(L^{-1}\ha{\chi}(L^{-1}\cdot))\|_{L_\ta^2 l_\xi^2},\\
    B&:=\sum_{L_1> L}L_1^{-\frac 12}
      \bigg\|\check{\psi}_{L_1}(\ta,\xi)
        \int_\R L^{-1}|(\ta-\ta')\ha{\chi}(L^{-1}(\ta-\ta'))\ti{u}(\ta',\xi)| d\ta' \bigg\|_{L_\ta^2 l_\xi^2}.
  }
  For $A$, the Young inequality in $\ta$ gives
  \EQQS{
    A\lesssim \sum_{L_1>L}L_1^{\frac 12}\|\psi_{L_1}\ti{u}\|_{L_\ta^2 l_\xi^2}
    \le \|u\|_{X^{0,\frac 12,1}}.
  }
  On the other hand, for $B$, as in the proof of Lemma \ref{lem_short1} we have
  \EQQS{
    B
    \lesssim L^{-1}\sum_{L_1 > L} L_1^{-\frac 12}
      \bigg\|\check{\psi}_{L_1}(\ta,\xi)
      \sum_{L_2} I_{L_2} L_2^{\frac 12}
        \|\psi_{L_2}(\ta',\xi)\ti{u}(\ta',\xi)\|_{L_{\ta'}^2} \bigg\|_{L_\ta^2 l_\xi^2},
  }
  where
  \EQQS{
    I_{L_2}:=\|(\ta-\ta')\ha{\chi}(L^{-1}(\ta-\ta'))\check{\psi}_{L_2}(\ta',\xi)\LR{\ta'-p_{\al+1}(\xi)}^{-\frac 12}\|_{L_{\ta'}^2}.
  }
  Now we divide the contribution of $I_{L_2}$ into two pieces.
  On one hand, when  $L_2\ll L_1$, we have $|\ta-\ta'|\ge |\ta-p_{\al+1}(\xi)|-|\ta'-p_{\al+1}(\xi)|\gtrsim L_1 $, which implies that $|\ta-\ta'|\sim L_1$.
  Since $\ha{\chi}\in \mathcal{S}(\R)$, we see that $\sup_{\ta\in\R}\LR{\ta}^4|\ha{\chi}(\ta)|\lesssim 1$.
  This and \eqref{eq_3.1} show that
  \EQQS{
    I_{L_2}
    &\lesssim L_1
      \|(L^{-1}(\ta-\ta'))^{-4}\check{\psi}_{L_2}(\ta',\xi)
      \LR{\ta'-p_{\al+1}(\xi)}^{-\frac 12}\|_{L_{\ta'}^2}\\
    &\lesssim L_1^{-3}L^4\|\check{\psi}_{L_2}(\ta',\xi)
      \LR{\ta'-p_{\al+1}(\xi)}^{-\frac 12}\|_{L_{\ta'}^2}
    \lesssim L_1^{-3}L^4.
  }
  On the other hand, when $L_1\lesssim L_2$, it holds that
  \EQQS{
    I_{L_2}
    \lesssim L_2^{-\frac 12}
      \|\ta'\ha{\chi}(L^{-1}\ta')\|_{L_{\ta'}^2}
    \lesssim L_1^{-\frac 12}L^{\frac{3}{2}} \; .
  }
  Combining the above  estimates, we get
  \EQQS{
    B
    &\lesssim L^{3}\sum_{L_1 > L} L_1^{-\frac{7}{2}}
      \bigg\|\ti{\psi}_{L_1}(\ta,\xi)
      \sum_{L_2\ll L_1} L_2^{\frac 12}
        \|\psi_{L_2}(\ta',\xi)\ti{u}(\ta',\xi)\|_{L_{\ta'}^2} \bigg\|_{L_\ta^2 l_\xi^2}\\
    &\quad+L^{\frac 12}\sum_{L_1>L} L_1^{-1}
      \bigg\|\ti{\psi}_{L_1}(\ta,\xi)
      \sum_{L_2\gtrsim L_1} L_2^{\frac 12}
        \|\psi_{L_2}(\ta',\xi)\ti{u}(\ta',\xi)\|_{L_{\ta'}^2} \bigg\|_{L_\ta^2 l_\xi^2}\\
    &\lesssim L^3\sum_{L_1>L}L_1^{-3}\|u\|_{X^{0,\frac 12,1}}
      +L^{\frac 12}\sum_{L_1>L}L^{-\frac 12}\|u\|_{X^{0,\frac 12,1}}
    \lesssim \|u\|_{X^{0,\frac 12,1}},
  }
  where we used \eqref{eq_3.1} in the second inequality.
  This completes the proof of \eqref{tyty}.
\end{proof}

\section*{Acknowledgments}

The second author is grateful to Prof.\ Kotaro Tsugawa for fruitful discussions on Remark \ref{rem_short}.
The second author was supported by the EPSRC New Investigator Award (grant no.\ EP/V003178/1), JSPS KAKENHI Grant Number JP20J12750, JSPS Overseas Challenge Program for Young Researchers and Iizuka Takeshi Scholarship Foundation.
A part of this work was conducted during a visit of the second author at Institut Denis Poisson (IDP) of Universit\'e de Tours in France. The second author is deeply grateful to IDP for its kind hospitality.


\begin{thebibliography}{99}
  \bibitem{ABFS89}
   L. Abdelouhab, J. L. Bona, M. Felland, and J.-C. Saut,
   Nonlocal models for nonlinear, dispersive waves,
   \textit{Phys.\ D},  \textbf{40} (1989), 360--392.

  \bibitem{BIT11}
   A. Babin, A. Ilyin, and E. Titi,
   On the regularization mechanism for the periodic Korteweg-de Vries equation,
   \textit{Comm.\ Pure Appl.\ Math.} \textbf{64} (2011), 591--648.

  \bibitem{B93}
   J. Bourgain,
   Fourier transform restriction phenomena for certain lattice subsets and applications to nonlinear evolution equations, part I: Schr\"odinger equation, part II: The KdV-equation,
   \textit{Geom.\ Funct.\ Anal.} \textbf{3} (1993), 107--156, 209--262.

  \bibitem{BGT1}
   N. Burq, P. G\'erard, and N. Tzvetkov,
   Strichartz inequalities and the nonlinear Schr\"odinger equation on compact manifolds,
   \textit{Amer.\ J. Math.} \textbf{126} (2004), 569--605.

  \bibitem{CKSTT04}
   J. Colliander, M. Keel, G. Staffilani, H. Takaoka, and T. Tao,
   Multilinear estimates for periodic KdV equations, and applications,
   \textit{J. Funct.\ Anal.} \textbf{211} (2004), 173--218.

  \bibitem{GKT}
   P. G\'erard, T. Kappeler, and P. Topalov,
   Sharp well-posedness results of the Benjamin-Ono equation in $H^s(\T,\R) $  and qualitative properties of its solution,
   arXiv:2004.04857.

  \bibitem{GKO13}
   Z. Guo, S. Kwon, and T. Oh,
   Poincar\'e-Dulac normal form reduction for unconditional well-posedness of the periodic cubic NLS,
   \textit{Comm.\ Math.\ Phys.} \textbf{322} (2013), 19--48.

  \bibitem{GLM14}
   Z. Guo, Y. Lin, and L. Molinet,
   Well-posedness in energy space for the periodic modified Benjamin-Ono equation,
   \textit{J. Differential Equations} \textbf{256} (2014), 2778--2806.

  \bibitem{H12}
   Z. Hani,
   A bilinear oscillatory integral estimate and bilinear refinements to Strichartz estimates on closed manifolds,
   \textit{Anal. PDE} \textbf{5} (2012), 339--363.

  \bibitem{HIKK}
   S. Herr, A. D. Ionescu, C. E. Kenig, and H. Koch,
   A para-differential renormalization technique for nonlinear dispersive equations,
   \textit{Comm.\ Partial Differential Equations} \textbf{35} (2010), 1827--1875.

  \bibitem{IKT}
   A. D. Ionescu, C. E. Kenig, and D. Tataru,
   Global well-posedness of the initial value problem for the KP-I equation in the energy space,
   \textit{Invent.\ Math.} \textbf{173} (2008), 265--304.

  \bibitem{KT06}
   T. Kappeler and T. Topalov,
   Global wellposedness of KdV in $ H^{-1}(\T,\R) $,
   \textit{Duke Math. J.} \textbf{135} (2006), 327--360.

   \bibitem{Kato83}
   T. Kato,
   On the Cauchy problem for the (generalized) Korteweg-de Vries equation,
   \textit{Stud.\ Appl.\ Math.} \textbf{8} (1983), 93--126.

  \bibitem{Kato95}
   T. Kato,
   On nonlinear Schr\"odinger equations. II. $H^s$-solutions and unconditional well-posedness,
   \textit{J. Anal.\ Math.} \textbf{67} (1995), 281--306. with Correction in  \textit{J. Anal.\ Math.} \textbf{68} (1996), 305.

  \bibitem{KPV1}
   C. E. Kenig, G. Ponce, and L. Vega,
   Well-posedness of the initial value problem for the Korteweg-de Vries equation,
   \textit{J. Amer.\ Math.\ Soc.} \textbf{4} (1991), 323--347.

  \bibitem{KPV2}
   C. E. Kenig, G. Ponce, and L. Vega,
   Well-posedness and scattering results for the generalized Korteweg-de Vries equation via the contraction principle,
   \textit{Comm.\ Pure Appl.\ Math.} \textbf{46} (1993), 527--620.

  \bibitem{KK}
   C. Kenig and K. Koenig,
   On the local well-posedness of the Benjamin-Ono and modified  Benjamin-Ono equations,
   \textit{Math.\ Res.\ Letters} \textbf{10} (2003), 879--895.

  \bibitem{KS21}
   K. Kim and R. Schippa,
   Low regularity well-posedness for generalized Benjamin-Ono equations on the circle,
   \textit{J. Hyperbolic Differ.\ Equ.} \textbf{18} (2021), 931--984.

  \bibitem{K19p}
   N. Kishimoto,
   Unconditional uniqueness for the periodic modified Benjamin-Ono equation by normal form approach,
   arXiv:1912.01363v2.

  \bibitem{K22}
   N. Kishimoto,
   Unconditional uniqueness for the periodic Benjamin-Ono equation by normal form approach,
   \textit{J. Math.\ Anal.\ Appl.} \textbf{514} (2022), Paper No. 126309, 22 pp.

  \bibitem{KTa}
   H. Koch and D. Tataru,
   A priori bounds for the 1D cubic NLS in negative Sobolev spaces,
   \textit{Int.\ Math.\ Res.\ Not.\ IMRN} (2007), Art.\ ID rnm053, 36 pp.

  \bibitem{KT1}
   H. Koch and N. Tzvetkov,
   On the local well-posedness of the Benjamin-Ono equation in $H^s(\R)$,
   \textit{Int.\ Math.\ Res.\ Not.\ IMRN} \textbf{26} (2003), 1449--1464.

  \bibitem{KO12}
   S. Kwon and T. Oh,
   On unconditional well-posedness of modified KdV,
   \textit{Int.\ Math.\ Res.\ Not.\ IMRN} \textbf{2012} Issue 15 (2012), 3509--3534.

  \bibitem{KOY20}
   S. Kwon, T. Oh, and H. Yoon,
   Normal form approach to unconditional well-posedness of nonlinear dispersive PDEs on the real line,
   \textit{Ann.\ Fac.\ Sci.\ Toulouse Math.} \textbf{29} (2020), 649--720.

  \bibitem{MN08}
   N. Masmoudi and K. Nakanishi,
   Energy convergence for singular limits of Zakharov type systems,
   \textit{Invent.\ Math.} \textbf{172} (2008), 535--583.

  \bibitem{MST}
   L. Molinet, J.-C. Saut, and N. Tzvetkov,
   Ill-posedness issues for the Benjamin-Ono and related equations,
   \textit{SIAM J. Math.\ Anal.} \textbf{33} (2001), 982--988.

  \bibitem{MR09}
   L. Molinet and F. Ribaud,
   Well-posedness in $H^1$ for generalized Benjamin-Ono equations on the circle,
   \textit{Discrete Contin.\ Dyn.\ Syst.} \textbf{23} (2009), 1295--1311.

  \bibitem{MV15}
   L. Molinet and S. Vento,
   Improvement of the energy method for strongly nonresonant dispersive equations and applications,
   \textit{Anal.\ PDE} \textbf{8} (2015), 1455--1495.

  \bibitem{MPV18}
   L. Molinet, D. Pilod, and S. Vento,
   On well-posedness for some dispersive perturbations of Burgers' equation,
   \textit{Ann.\ I. H. Poincar\'e Anal.\ Non Lin\'eaire} \textbf{35} (2018), 1719--1756.

  \bibitem{MPV19}
   L. Molinet, D. Pilod, and S. Vento,
   On unconditional well-posedness for the periodic modified Korteweg-de Vries equation,
   \textit{J. Math.\ Soc.\ Japan} \textbf{71} (2019), 147--201.

  \bibitem{MT22}
   L. Molinet and T. Tanaka,
   Unconditional well-posedness for some nonlinear periodic one-dimensional dispersive equations,
   \textit{J. Funct. Anal.} \textbf{283} (2022), 109490.

  \bibitem{MT22p}
   L. Molinet and T. Tanaka,
   On well-posedness of some one-dimensional periodic dispersive equations with analytical nonlinearity,
   \textit{RIMS K\^{o}ky\^{u}roku Bessatsu} \textbf{B95} (2024), 53--72.

  \bibitem{MPp}
   R. Mosincat and D. Pilod,
   Unconditional uniqueness for the Benjamin-Ono equation,
   \textit{Pure Appl.\ Anal.} \textbf{5} (2023), 285--322.

  \bibitem{P21}
   J. Palacios,
   Local well-posedness for the gKdV equation on the background of a bounded function,
   \textit{Rev.\ Mat.\ Iberoam.} \textbf{39} (2023), 341--396.

  \bibitem{Po1}
   G. Ponce,
   On the global well-posedness of the Benjamin-Ono equation,
   \textit{Diff.\ Int.\ Eq.} \textbf{4} (1991), 527--542.

  \bibitem{ST01}
   J.-C.\ Saut and N. Tzvetkov,
   On periodic KP-I type equations,
   \textit{Comm.\ Math.\ Phys.} \textbf{221} (2001), 451--476.

  \bibitem{SS}
   R. Schippa and R. Schnaubelt,
   Strichartz estimates for Maxwell equations in media: the fully anisotropic case,
   \textit{J. Hyperbolic Differ.\ Equ.} \textbf{20} (2023), 917--966.

   \bibitem{Tao04}
   T. Tao,
   Global well-posedness of the Benjamin-Ono equation in $H^1(\R)$,
   \textit{J. Hyperbolic Differ.\ Equ.} \textbf{1} (2004), 27--49.

  \bibitem{Tao}
   T. Tao,
   \textit{Nonlinear Dispersive Equations; Local and Global Analysis},
   CBMS, 2006.
\end{thebibliography}
\end{document}